\documentclass[10pt]{amsart}

\usepackage{amssymb, amscd}
\usepackage{epsfig, mathtools}
\usepackage[vmargin=1in, hmargin=1.3in]{geometry}
\usepackage[font=small,format=plain,labelfont=bf,up,textfont=it,up]{caption}
\usepackage{microtype}
\usepackage[shortlabels]{enumitem}
\usepackage[backref=page, bookmarks, bookmarksdepth=2, colorlinks=true, linkcolor=blue, citecolor=blue, urlcolor=blue]{hyperref}
\usepackage{etoolbox}
\usepackage{tikz-cd}
\usepackage{mathrsfs}  
\usepackage{float}
\usepackage{import}

\usepackage{caption}
\usepackage{transparent}
\usepackage[inkscapelatex=true]{svg}

\apptocmd{\thebibliography}{\raggedright}{}{}

\makeatletter
\patchcmd{\@maketitle}{\global\topskip42\p@\relax}
  {\global\topskip42\p@\relax \vspace*{-38pt}}
  {}{}
\makeatother

\setenumerate[0]{label=(\alph*)}

\renewcommand*{\backref}[1]{}
\renewcommand*{\backrefalt}[4]{%
    \ifcase #1 (Not cited.)%
    \or        (Cited on page~#2.)%
    \else      (Cited on pages~#2.)%
    \fi}

\numberwithin{equation}{section}

\theoremstyle{plain}
\newtheorem{theorem}{Theorem}[section]
\newtheorem{maintheorem}{Theorem}

\newtheorem{proposition}[theorem]{Proposition}
\newtheorem{lemma}[theorem]{Lemma}

\newtheorem{corollary}[theorem]{Corollary}

\newtheorem{question}[theorem]{Question}

\theoremstyle{definition}
\newtheorem{asm}[theorem]{Assumption}

\newtheorem{defn}[theorem]{Definition}

\newtheorem{notn}[theorem]{Notation}

\theoremstyle{remark}
\newtheorem{rmk}[theorem]{Remark}

\newtheorem{eg}[theorem]{Example}

\DeclareMathOperator{\Mod}{Mod}

\title{Flexible Surfaces in $\mathbb{C}P^2$ and $S^2 \times S^2$}

\author{Joshua Lehman}
\address{Dept of Mathematics; Univ of Notre Dame; Hurley Hall; Notre Dame, IN 46556; USA}
\email{jlehman4@nd.edu}
\thanks{}

\date{\today}

\begin{document}
\maketitle

\begin{abstract}
A surface $\Sigma$ in a 4-manifold $M$ is called flexible if any mapping class of the surface arises as the restriction of a diffeomorphism $(M,\Sigma) \to (M,\Sigma)$. We construct flexible surfaces in $\mathbb{C}P^2$ and $S^2 \times S^2$ within any prescribed non-characteristic homology class. Within characteristic homology classes there is a spin structure obstructing flexibility and we construct so-called spin-flexible representatives.
\end{abstract}

\section{Introduction}

Let $\Sigma^2 \subset M^4$ be a closed oriented surface in a closed oriented 4-manifold $M$. Associated to the pair $(M,\Sigma)$ is the \textit{relative mapping class group} $\Mod(M,\Sigma) = \pi_0(\textrm{Diff}^+(M,\Sigma))$, where $\textrm{Diff}^+(M,\Sigma)$ consists of orientation preserving diffeomorphisms of $M$ fixing $\Sigma$ \textit{setwise}. Restriction determines a homomorphism $\mathcal{R} : \Mod(M,\Sigma) \to \Mod^{\pm}(\Sigma) = \pi_0(\textrm{Diff}(\Sigma))$. Following Hirose \cite{HiroseS^4}, the \textit{extendable subgroup} $\mathcal{E}(M,\Sigma)$ of $\Sigma$ in $M$ is the image of $\mathcal{R}$ in $\Mod(\Sigma) = \pi_0(\textrm{Diff}^+(\Sigma))$. The surface $\Sigma$ is called \textit{flexible} in $M$ if $\mathcal{E}(M,\Sigma) = \Mod(\Sigma)$.  If $\Sigma$ is \textit{characteristic}\footnote{Recall that a class in $ H^2(M;\mathbb{Z})$ is \textit{characteristic} if it is an integral lift of the second Stiefel-Whitney class of $M$. A closed surface $\Sigma \subset M$ is characteristic if the Poincar\'{e} dual of $[\Sigma] \in H_2(M;\mathbb{Z})$ is.} in $M$ and $H_1(M;\mathbb{Z}) =0$ there is an obstruction to flexibility, namely \textit{Rochlin's quadratic form}  $q_{\Sigma} : H_1(\Sigma;\mathbb{Z}/2\mathbb{Z}) \to \mathbb{Z}/2\mathbb{Z}$ (Proposition \ref{RochlinsForm}). The mapping class group $\Mod(\Sigma)$ acts on the set of quadratic forms on $H_1(\Sigma;\mathbb{Z}/2\mathbb{Z})$ by pullback. By definition of $q_\Sigma$, $\mathcal{E}(M,\Sigma) \leq \Mod(\Sigma)[q_{\Sigma}]$, where $\Mod(\Sigma)[q_{\Sigma}]$ denotes the stabilizer of $q_{\Sigma}$. We say $\Sigma$ is \textit{spin-flexible} if $\mathcal{E}(M,\Sigma) = \Mod(\Sigma)[q_{\Sigma}]$. A guiding question of Hirose, appearing as Problem 6.3 in \cite{HiroseSurvey}, is as follows. 

\begin{question} \label{homology}
    Assume that $M$ is simply connected. Let $x \in H_2(M;\mathbb{Z})$ be an element which is not characteristic. Does $x$ admit a flexible representative? 
\end{question}
    
    We resolve this question for $\mathbb{C}P^2$ and $S^2 \times S^2$. In particular we prove

\begin{maintheorem}
Let $M = \mathbb{C}P^2$ or $S^2 \times S^2$. Fix $x \in H_2(M;\mathbb{Z})$. If $x$ is characteristic, there exists a spin-flexible representative for $x$, otherwise there is a flexible representative for $x$. 
\end{maintheorem}

We may further ask for flexible representatives which minimize genus in their homology class. Recall that smooth degree $d$ plane curves in $\mathbb{C}P^2$ represent $d[\mathbb{C}P^1] \in H_2(\mathbb{C}P^2;\mathbb{Z})$, minimize genus in their homology class (Thom Conjecture, \cite{KronMrowka}), and are characteristic if and only if $d$ is odd. In this direction we establish that a single \textit{stabilization} suffices for (spin) flexibility. 

\begin{maintheorem} \label{ThmB}
Let $X_d \subset \mathbb{C}P^2$ be a smooth plane curve of degree $d$.  Let $T \subset S^4$ be the boundary of a solid torus in $S^4$. Let $Z_d$ be the stabilization $X_d \# T \subset \mathbb{C}P^2 \# S^4 \cong \mathbb{C}P^2$. If $d$ is even, $Z_d$ is flexible. If $d$ is odd, $Z_d$ is spin-flexible. 
\end{maintheorem}

This result may be viewed as a generalization of \cite[Theorem 5.1]{HiroseCP^2} to all $d \neq 3$ (Remark \ref{d=3difference}).\newline
 
\textbf{Context and Related Work.} Recall that a surface in a 4-manifold $M$ is \textit{unknotted} if it is the boundary of a handlebody\footnote{In this paper we exclusively work with orientable surfaces.} in $M$. Montesinos established \cite[Theorem 5.4]{Montesinos} that an unknotted torus in $S^4$ is spin-flexible and Hirose \cite[Theorem 1.1]{HiroseS^4} generalized this to unknots of all higher genera. On the other hand, there are examples (non-trivial spun tori) of \textit{knotted surfaces} in $S^4$ whose extendable subgroup is much \textit{smaller} (namely of infinite index, see \cite{Hirosespunknots} for precise details). In $\mathbb{C}P^2$ unknots are no longer characteristic and are \textit{flexible} surfaces \cite[Theorem 3.1]{HiroseCP^2}. A broader version of Question \ref{homology} asks only for the \textit{existence} of flexible surfaces in a simply connected 4-manifold. In this regard, Hirose and Yasuhara construct flexible surfaces of any genus in $\mathbb{C}P^2, S^2\times S^2, \mathbb{C}P^2 \# \overline{\mathbb{C}P^2}$ or any elliptic surface \cite[Theorem 3.1]{HiroseYasuhara}. Yet another version allows one to stabilize the ambient 4-manifold with an $S^2 \times S^2$ summand. Hirose and Yasuhara prove \cite[Theorem 4.1]{HiroseYasuhara} that for any $\Sigma^2 \subset M^4$ with $M$ simply connected, the surface $\Sigma \# S^2 \times \{*\} \subset M \# S^2 \times S^2$ is flexible. In particular, if one is allowed to change \textit{both} the ambient manifold and homology class, the best result is possible.

 Let $X_d \subset \mathbb{C}P^2$ be a smooth degree $d$ plane curve. Hirose proved that $X_3$ and $X_4$ are flexible surfaces in $\mathbb{C}P^2$ \cite[Theorem 4.2]{HiroseCP^2}. Moreover, Hirose proved that, for an unknot $\Sigma_g \subset S^4$ of genus $g$, the surface $X_3 \# \Sigma_g$ is spin-flexible \cite[Theorem 5.1]{HiroseCP^2}. A corollary of work of Salter is that $X_5 \subset \mathbb{C}P^2$ is spin-flexible \cite[Theorem A]{SalterPlaneCurves}. The situation for plane curves of higher degree is, to the best of the authors knowledge, unknown. However, $\mathcal{E}(\mathbb{C}P^2,X_d)$ is a finite index subgroup of $\Mod(X_d)$ for $d\geq 6$. This follows from work of Salter \cite[Theorem 1.3]{SalterToric}.

Precisely, let $\mathcal{P}_d \subset |\mathcal{O}(d)|$ denote the locus of smooth degree $d$ plane curves. The tautological family over $\mathcal{P}_d$ determines a monodromy representation $\rho_d : \pi_1(\mathcal{P}_d,X_d) \to \Mod(X_d)$. By the isotopy extension theorem, the homomorphism $\rho_d$ factors through an \textit{embedded monodromy representation} $\widetilde{\rho}_d : \pi_1(\mathcal{P}_d,X_d) \to \Mod(\mathbb{C}P^2,X_d)$. In particular, the image of $\rho_d$ is a lower bound for $\mathcal{E}(\mathbb{C}P^2,X_d)$. Extending work in \cite{SalterToric}, Calderon-Salter \cite{CalderonSalter} establish that the image of $\rho_d$ in $\Mod(X_d)$ coincides with the \textit{stabilizer} $\Mod(X_d)[\phi_d]$ of a $(d-3)$-spin structure $\phi_d$.\footnote{An $r$-spin structure on a surface $\Sigma$ is a cohomology class $\xi$ on the unit tangent bundle $\xi \in H^1(UT\Sigma;\mathbb{Z}/r\mathbb{Z})$ which assigns 1 to a positively oriented fiber.} In particular, $$ \Mod(X_d)[\phi_d] \leq \mathcal{E}(\mathbb{C}P^2,X_d).$$

The higher spin structure $\phi_d$ is an instance of a more general phenomena used in proving both Theorems \ref{thmA} and \ref{ThmB} - it arises from a specific \textit{framing}\footnote{Trivialization of the tangent bundle.} of the surface $X_d$ in an affine piece $\mathbb{C}^2 \subset \mathbb{C}P^2$. Let $B^4_{\varepsilon}(0) \subset \mathbb{C}^2$ be a ball of radius $\varepsilon > 0$ about the origin $0 \in \mathbb{C}^2$. Up to isotopy, $X_d \cap B^4_{\varepsilon}(0) \subset \mathbb{C}^2$ is a Seifert surface for the closure of a positive braid. Ferretti \cite{Ferretti}, using work of Stallings \cite{Stallings}, associates to the closure of any positive braid $\beta$ a \textit{framed} Seifert surface $F_\beta \subset S^3$. Our representatives are obtained by capping off such Seifert surfaces. In fixing only the ambient homology class we may stabilize with an unknotted surface. Our techniques rest on the contrast between the geometric rigidity of the framed surfaces and the methods of Hirose established for unknots in $S^4$ \cite{HiroseS^4}.

\begin{rmk} (Higher spin structures and Rochlin's form)
When $d$ is odd, this story is not disjoint from the topological one. In this case, $d-3$ is even and $\phi_d$ has a mod 2 reduction $\overline{\phi_d}$, which \textit{coincides} with the 2-spin structure corresponding to Rochlin's quadratic form (see \cite{Libgober} for details) and we have 
$$ \Mod(X_d)[\phi_d] \leq \mathcal{E}(\mathbb{C}P^2,X_d) \leq \Mod(X_d)[\overline{\phi_d}]. $$
\end{rmk}

\textbf{General overview.} Consider $M = \mathbb{C}P^2$ or $S^2 \times S^2$, equipped with its standard handlebody decomposition. Recall that any homology class $x \in H_2(M;\mathbb{Z})$ admits a representative $\Sigma \subset M$ which intersects the 0-handle $B^4 \subset M$ in a Seifert surface $F_\Sigma = \Sigma \cap B^4 = \Sigma \cap \partial B^4$ where $\Sigma\backslash \textrm{int}(F_\Sigma) = D_\Sigma$ is a collection of closed 2-disks (Subsection \ref{ConstructingSurfaces}). Let $\theta_{F_\Sigma} : H_1(F_\Sigma;\mathbb{Z}/2\mathbb{Z}) \to \mathbb{Z}/2\mathbb{Z}$ be the Seifert form and $i_\Sigma : F_\Sigma \to \Sigma$ inclusion.  We begin by assuming that the surface $\Sigma$ is characteristic.

Let $q_\Sigma : H_1(\Sigma;\mathbb{Z}/2\mathbb{Z}) \to \mathbb{Z}/2\mathbb{Z}$ denote Rochlin's quadratic form. Then $i_{\Sigma}^*q_\Sigma = \theta_{F_\Sigma}$ (Proposition \ref{Rochlinislink}). Note that $\theta_{F_\Sigma}$ is always defined. In general, the failure of $\theta_{F_\Sigma}$ to extend over $H_1(\Sigma;\mathbb{Z}/2\mathbb{Z})$ is a witness for the failure of $\Sigma$ to be characteristic. This is ultimately the source of flexibility (we refer the reader to Figure \ref{slide} which constructs a flexible representative for $2[\mathbb{C}P^1] \in H_2(\mathbb{C}P^2;\mathbb{Z}))$. 

According to Calderon-Salter \cite[Proposition 5.1]{CalderonSalter}, if the genus of $\Sigma$ is at least 3, the stabilizer \[\Mod(\Sigma)[q_\Sigma] = \mathcal{T}_{q_\Sigma} := \langle T_{\gamma} \,|\, \gamma \subset \Sigma \textrm{ nonseparating simple closed curve}, q_\Sigma(\gamma) = 1 \rangle.\] To prove Theorem \ref{thmA} we construct a representative $\Sigma$ for which $\mathcal{T}_{q_\Sigma} \leq \mathcal{E}(M,\Sigma)$. A consequence of this approach is that one must deal with \textit{arbitrary} curves. This is done in two steps. 

First we establish a lower bound for $\mathcal{E}(M,\Sigma)$ which supports a \textit{change-of-coordinates} principle\footnote{Change-of-coordinates refers to a family of results allowing one to exchange configurations of curves on a surface through diffeomorphisms of the surface provided the configurations have the same `combinatorics'. }. It is the image of a \textit{framed} mapping class group (Definition \ref{framedmcg}). Second, we stabilize with an unknotted surface in order to manufacture more extendable mapping classes.

Precisely, let $S = F_\Sigma \# \Sigma_g \subset S^3$, where $\Sigma_g \subset S^3$ is an unknotted surface of genus $g$. Capping $S$ off with $D_\Sigma$ yields a representative $Z = S \cup D_\Sigma$. We prove that for $g \geq 5$ and $\gamma \subset S$ a nonseparating simple closed curve with \[\theta_S(\gamma) = \mathrm{lk}(\gamma,\gamma^+) =  1 \mod 2,\] one has $T_{\gamma} \in \mathcal{E}(M,Z)$ (Lemma \ref{twistson1curves}). As $Z\backslash S = D_\Sigma$ is a collection of disks and $i_Z^*q_Z = \theta_S$, it follows that \[\mathcal{T}_{q_Z} \leq \mathcal{E}(M,Z).\]

Assume now that $x$ is not a characteristic element. Lemma \ref{twistson1curves} still applies, needing only the pair $(B^4,S)$. To establish flexibility, it is enough to prove that for any nonseparating simple closed curve $\delta \subset S$ with \[\theta_S(\delta) = \mathrm{lk}(\delta,\delta^+) = 0 \mod 2\] one has $T_{\delta} \in \mathcal{E}(M,Z)$. We show that isotoping $\delta$ across a disk in $D_\Sigma$ yields a curve $\widetilde{\delta}$ on $S$ with $\theta_S(\widetilde{\delta}) = 1 \mod 2$ (Proposition \ref{geometricwitness}). By Lemma \ref{twistson1curves},  $T_{\widetilde{\delta}} \in \mathcal{E}(M,Z)$. On $Z$, the curves $\delta$ and $\widetilde{\delta}$ are isotopic and so $T_{\delta} \in \mathcal{E}(M,Z)$. As $\Mod(Z)$ is generated by Dehn twists on nonseparating simple closed curves, $Z$ is flexible. \newline

\textbf{Outline of the paper.} In Section 2 we recall foundational material regarding extendable mapping classes. We do so in detail for we conclude the section by showing that the methods of Hirose established in \cite{HiroseS^4} prove spin-flexibility for properly embedded unknots in the 4-ball. In Section 3 we recall relevant background regarding \textit{framed} and \textit{spin} mapping class groups. In Section 4 we prove Theorem \ref{thmA}. Starting with a family of framed Seifert surfaces for homogeneous braid closures studied in \cite{Ferretti}, we use techniques from Section 3 to establish, after stabilizing with an unknot of sufficiently large genus, that the Dehn twist on any nonseparating simple closed curve with odd self-linking number is an extendable mapping class (Lemma \ref{twistson1curves}). Thereafter we readily construct (spin) flexible representatives in $\mathbb{C}P^2$ and $S^2 \times S^2$. Flexibility for non-characteristic classes is a consequence of Proposition \ref{geometricwitness}. In Section 5 we prove Theorem \ref{ThmB}. Methods from Section 4 are adapted to the case of a \textit{single} stabilization, using the \textit{geometric monodromy} associated to a positive braid identified in \cite{Ferretti}. \newline

\textbf{Conventions and notation.} We work in the smooth category throughout. All manifolds are tacitly oriented, and all submanifolds are embedded. For a simple closed curve $\gamma$ on a surface $\Sigma$, denote the left handed Dehn twist along $\gamma$ by $T_{\gamma} :\Sigma \to \Sigma$. We conflate curves on surfaces, as well as diffeomorphisms, with their associated isotopy classes. For a pair of curves $\alpha,\beta \subset \Sigma$, we denote their geometric intersection number by $i(\alpha,\beta)$. Given a manifold $M$ and subspace $A \subset M$, let $\textrm{Diff}_\partial(M,A)$ denote diffeomorphisms of $M$ fixing $A$ \textit{setwise}, and restricting to the identity on $\partial M$. For a submanifold $S \subset M$, we denote the normal bundle of $S$ in $M$ by $\nu_{S\subset M}$, implicitly identified with an open neighborhood of $S$ in $M$. \newline

\textbf{Acknowledgements.} I thank \.{I}nan\c{c} Baykur for useful conversations regarding the extendable subgroup. I thank my advisor Nick Salter for extensive feedback and discussion, as well as his generosity in both of these matters. I would also like to thank Aru Mukherjea for comments on an early draft.

\section{Extension Methods and Rochlin's Quadratic Form}

We recall the principal techniques for extending mapping classes of surfaces in 4-manifolds and the obstructions that arise in doing so. We conclude by noting that properly embedded unknots in the 4-ball are spin-flexible rel boundary (Proposition \ref{relboundaryhirose}). \newline

\textbf{Extendable subgroups.} Let $(\Sigma^2,\partial \Sigma) \subset (M^4,\partial M)$ be a properly embedded oriented compact surface in a oriented 4-manifold $M$. Let $\Mod(M,\Sigma) = \pi_0(\textrm{Diff}^+_{\partial}(M,\Sigma))$ be the smooth \textit{relative mapping class group} of the pair $(M,\Sigma)$. Restriction determines a homomorphism \[\mathcal{R}: \Mod(M,\Sigma) \to \Mod^{\pm}(\Sigma) = \pi_0(\textrm{Diff}_\partial(\Sigma,\partial\Sigma)).\]

\begin{defn} Define $\mathcal{E}_{\partial}(M,\Sigma)$ to be the image of $\mathcal{R}$ in $\Mod(\Sigma) = \pi_0(\textrm{Diff}^+_\partial(\Sigma,\partial\Sigma)) \leq \Mod^{\pm}(\Sigma)$. 
\end{defn}

\begin{rmk}
Several calculations of $\mathcal{E}(M,\Sigma)$ in the literature (\cite{HiroseCP^2}, \cite{HiroseS^4}) construct extensions of mapping classes by isotopies of $\Sigma$ in $M$. In particular they actually compute the \textit{smooth monodromy group} of $\Sigma$ in $M$. The same is done throughout this paper. In the presence of boundary all isotopies are rel boundary. 
\end{rmk}

\begin{rmk}
That the target of restriction $\mathcal{R} : \Mod(M,\Sigma) \to \Mod^\pm(\Sigma)$ is the \textit{extended} mapping class group is a feature. For example, complex conjugation on $\mathbb{C}P^2$ yields an orientation reversing diffeomorphism on a smooth plane curve (defined by a homogeneous polynomial with real coefficients). Of course, complex conjugation is still an orientation preserving diffeomorphism of $\mathbb{C}P^2$. On the other hand, any isotopy of a surface in $M$ must (after isotopy extension) preserve orientation and act trivially on the homology of $M$. 
\end{rmk}

\subsection{Rochlin's quadratic form.} A \textit{membrane} for the pair $(M,\Sigma)$ is an embedded oriented surface $F \subset \textrm{int}(M)$, with $\partial F \subset \Sigma$ a simple closed curve and interior transverse to $\Sigma$. If $H_1(M;\mathbb{Z}) = 0$ any curve on $\Sigma$ has a membrane. In extracting obstructions to extending mapping classes, a guiding principle is that invariants associated to membranes must be preserved by $\mathcal{E}_{\partial}(M,\Sigma)$ (this avenue is pursued in \cite{SlopeGenera}, for example, by recording the minimum genus of amongst membranes along a simple closed curve). 

Given a membrane $F$ for the pair $(M,\Sigma)$, the \textit{relative Euler number} $e_{\Sigma}(F) \in \mathbb{Z}$ is defined as follows. Up to isotopy, we may assume that $F$ meets $\Sigma$ normally along $\partial F$. Let $\nu_F \to F$ denote the normal bundle of $F$ in $M$. As $F$ is homotopy equivalent to a wedge of circles and $\nu_F$ is an orientable 2-plane bundle, there exists a trivialization $\nu_F \cong F \times \mathbb{R}^2$. The image of a nowhere vanishing section of the normal bundle $\nu_{\partial F \subset \Sigma} $ sits inside $\nu_F$ and the trivialization then determines a smooth map $\partial F \to \mathbb{R}^2 - \{0\}$. The degree of this map is the relative Euler number $e_{\Sigma}(F) \in \mathbb{Z}$. We remind the reader that the relative Euler number itself is not an invariant of $\partial F$ by the \textit{spinning} construction, which changes $e_{\Sigma}(F)$ at the expense of creating intersections between $\textrm{int}(F)$ and $\Sigma$ (see \cite{FreedmanKirby}). 

We record the presence of Rochlin's form in the \textit{rel boundary} setting.  Recall that a \textit{quadratic form} $q : H_1(\Sigma;\mathbb{Z}/2\mathbb{Z}) \to \mathbb{Z}/2\mathbb{Z}$ is a function satisfying 
\[ 	q(x+y) = q(x)+q(y) + \langle x,y\rangle \]
for all $x,y \in H_1(\Sigma;\mathbb{Z}/2\mathbb{Z})$. Here $\langle -,-\rangle$ denotes the mod 2 intersection pairing, arising from the homomorphism $H_1(\Sigma;\mathbb{Z}/2\mathbb{Z}) \to H_1(\Sigma,\partial \Sigma;\mathbb{Z}/2\mathbb{Z})$ induced by inclusion and Poincar\'{e} duality.
The mapping class group $\Mod(\Sigma)$ acts on the set of quadratic forms on $H_1(\Sigma;\mathbb{Z}/2\mathbb{Z})$ by pullback $(f\cdot q)(x) := q(f^{-1}_*(x))$.

\begin{defn}
Let $q : H_1(\Sigma;\mathbb{Z}/2\mathbb{Z}) \to \mathbb{Z}/2\mathbb{Z}$ be a quadratic form. The stabilizer of $q$, denoted $\Mod(\Sigma)[q]$, is
\[ \Mod(\Sigma)[q] = \{f \in \Mod(\Sigma)\,|\, (f\cdot q) = q\}. \]
\end{defn}

\begin{proposition} \label{RochlinsForm}
Suppose that $H_1(M;\mathbb{Z}) = 0$ and  $\Sigma \subset M$ is characteristic, that is, $[\Sigma,\partial\Sigma] \in H_2(M,\partial M;\mathbb{Z}/2\mathbb{Z})$ is dual to $w_2(M) \in H^2(M;\mathbb{Z}/2\mathbb{Z})$ under Lefschetz duality. Then \textit{Rochlin's quadratic form}  $q_{\Sigma} : H_1(\Sigma;\mathbb{Z}/2\mathbb{Z}) \to \mathbb{Z}/2\mathbb{Z}$  is defined and \[\mathcal{E}(M,\Sigma) \leq \Mod(\Sigma)[q_{\Sigma}].\]
\end{proposition}

\begin{proof}
We briefly recall some of the details, referring the reader to \cite[Lemma 2]{Klug}. For a simple closed curve $\gamma \subset \Sigma$, pick a membrane $F$ with $\partial F = \gamma$. Then \[q_{\Sigma}(\gamma) = |\textrm{int}(F) \cap \Sigma| + e_{\Sigma}(F) \mod 2.\] The condition that $\Sigma$ be characteristic in $M$ ensures this is independent of the choice of membrane for $\gamma$. Transporting membranes with a diffeomorphism $(M,\Sigma) \to (M,\Sigma)$ shows that $\mathcal{E}(M,\Sigma) \leq \Mod(\Sigma)[q_{\Sigma}]$.
\end{proof}

\begin{defn}
We say that $\Sigma \subset M$ is \textit{flexible rel boundary} if $\mathcal{E}_{\partial}(M,\Sigma) = \Mod(\Sigma)$. A characteristic surface $\Sigma \subset M$ is  \textit{spin flexible rel boundary} if $\mathcal{E}_{\partial}(M,\Sigma) = \Mod(\Sigma)[q_\Sigma]$.
\end{defn}

\textbf{The Arf invariant}. Provided the intersection pairing on $H_1(\Sigma;\mathbb{Z}/2\mathbb{Z})$ is nondegenerate, quadratic forms on $H_1(\Sigma;\mathbb{Z}/2\mathbb{Z})$ are in the same orbit under $\Mod(\Sigma)$ if and only if they have the same \textit{Arf invariant}, a $\mathbb{Z}/2\mathbb{Z}$-valued invariant of a quadratic form (see \cite{Randal-Williams} for details). Quadratic forms with Arf invariant one (resp. zero) are called \textit{odd} (resp. \textit{even}). In the closed setting the Arf invariant captures ambient topology as follows. 

\begin{theorem}[Freedman-Kirby \cite{FreedmanKirby}] \label{ArfRochlin}
Assume that $M$ is closed, simply connected and that $\Sigma \subset M$ is characteristic. Then 
\[ \mathrm{Arf}(q_\Sigma) = \frac{1}{8}(\sigma(M) - \Sigma\cdot\Sigma) \mod 2,\]
where $\sigma(M)$ denotes the signature of $M$, and $\Sigma\cdot \Sigma$ the self-intersection of $\Sigma$ in $M$. 
\end{theorem}

\begin{rmk}
Whilst we work exclusively in the smooth category, we remark that Rochlin's form can be defined in the topologically locally flat category (see \cite{Klug} for discussion on this matter). 
\end{rmk}

\textbf{Rochlin's form in 3$\frac{1}{2}$ dimensions}. We record a well known perspective on Rochlin's form for surfaces which interpolate between 3 and 4 dimensions. Recall that, for a Seifert surface $S \subset S^3 = \partial B^4$, there is the self-linking form $\theta_S: H_1(S;\mathbb{Z}/2\mathbb{Z}) \to \mathbb{Z}/2\mathbb{Z}$, defined on primitive elements by \[\theta_S(\gamma) = \mathrm{lk}(\gamma,\gamma^+) \mod 2\] where $\gamma^+$ denotes a positive push-off of a curve $\gamma \subset S$, and $\mathrm{lk}$ denotes the linking number of the 1-cycles in $S^3$. 

\begin{proposition}[Page 117 of \cite{Saveliev}]
Let $\Sigma^2 \subset M^4$ be a characteristic surface, $H_1(M;\mathbb{Z}) = 0$ and $B^4 \subset M$ an embedded 4-ball. Assume that $ \Sigma \cap B^4 = \Sigma \cap \partial B^4 = S$ a Seifert surface for a link. Let $i : S \to \Sigma$ be inclusion and $\theta_S$ the linking form on $S$. Then $i^*q_{\Sigma} = \theta_S$ where $q_\Sigma : H_1(\Sigma;\mathbb{Z}/2\mathbb{Z}) \to \mathbb{Z}/2\mathbb{Z}$ denotes Rochlin's form. 
\label{Rochlinislink}
\end{proposition}

The above is a statement of equality on the nose. We record an instance in which the linking form and Rochlin's form, whilst not defined on the same surface, may still be identified. 

\begin{proposition} \label{pullbackRochlin}
For $S \subset S^3 = \partial B^4$ a Seifert surface, let $S'$ denote the result of pushing the interior of $S$ into the interior of $B^4$. The surface $(S',\partial S') \subset (B^4,S^3)$ is characteristic and Rochlin's form $q_{S'}$ is defined on $S'$. The isotopy used to push the surface into $B^4$ yields an identification $S' \cong S$, and the pullback of $\theta_S$ to $S'$ coincides with Rochlin's form $q_{S'}$. \end{proposition}

\begin{proof}
The argument is identical to the proof of Proposition \ref{Rochlinislink} in \cite[Page 117]{Saveliev}. 
\end{proof}

\subsection{Extension methods.} We now turn to examples of elements in $\mathcal{E}_\partial(M,\Sigma)$ for pairs $(M,\Sigma)$. \newline

\textbf{Tube isotopy.} We begin with a quintessential and well-known isotopy of a surface rendering a non-trivial extendable mapping class (\cite{Hirosespunknots}). Consider a standard annulus $F$ in $S^3 \times \{1/2\} \subset S^3 \times [0,1]$ with core curve $c$. Depicted in Figure \ref{squaretwist} is an isotopy of $F$ in $S^3 \times [0,1]$ (rel boundary) which induces the square of a (right-handed) Dehn twist along $c$. \newline 
\begin{figure}
    \centering
    \includegraphics[width=0.9\linewidth]{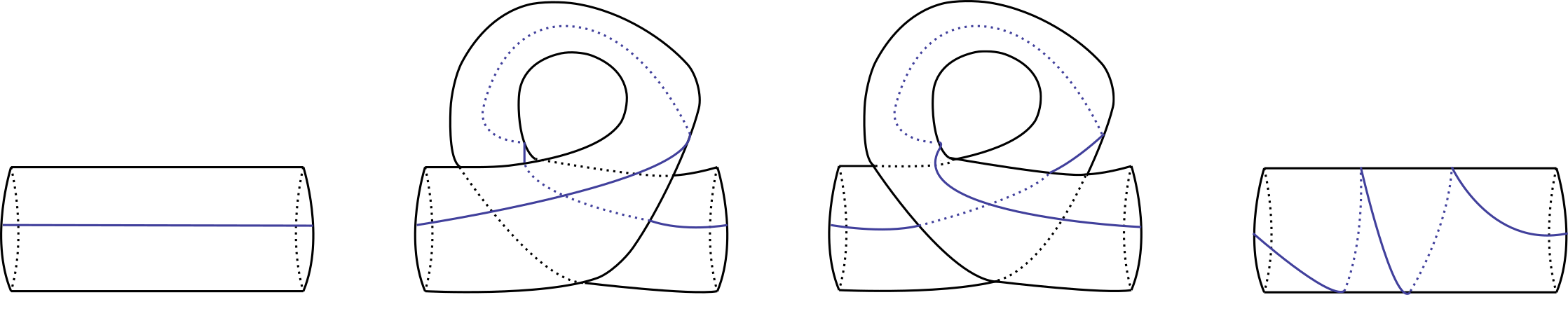}
    \caption{\textnormal{A well-known isotopy of a standard annulus in $S^3 \times I$ rel boundary inducing the square of a (right handed) twist along a core curve. }}
    \label{squaretwist}
\end{figure}

\textbf{Hopf band curves and fibered subsurfaces.} Let $B^4$ denote the closed unit ball in $\mathbb{R}^4$,  $S^3 = \partial B^4$. Let $A$ be an unknotted annulus in $S^3$ with one full positive or negative twist, so $A$ is a positive or negative \textit{Hopf band}. By pushing the interior of $A$ into $B^4$, we obtain a properly embedded surface $A' \subset B^4$. Let $a$ denote the core curve of the annulus $A'$. A key technical result, establish by Hirose in \cite[Proposition 2.1]{HiroseCP^2}, is

\begin{proposition} \label{HiroseHopfBand}
We have $T_a \in \mathcal{E}_{\partial}(B^4,A')$. In fact, $T_a$ is a smooth monodromy element for the pair $(B^4,A')$, that is, exists a diffeomorphism $T \in \mathrm{Diff}_{\partial}^+(B^4) $, smoothly isotopic to the identity rel boundary for which $T\rvert_{A'} = T_a$. 
\end{proposition}

We record the prototypical application of Proposition \ref{HiroseHopfBand}. Let $S \subset S^3$ be a Seifert surface for a link.  
\begin{defn} \label{Hopfbandcurve}
 We call a curve $\gamma \subset S$ a \textit{Hopf band curve} if $\overline{\nu_{\gamma \subset S}}$ is a Hopf band in $S^3$. In particular, $\gamma$ is an unknot in $S^3$ and, after orienting $\gamma$, $\textrm{lk}_S(\gamma,\gamma^+) = \pm 1$ where $\gamma^+$ denotes a positive push-off of $\gamma \subset S$. 
 \end{defn}

 Push the interior of $S$ into $\textrm{int}(B^4)$ to obtain a properly embedded surface $S'$. Explicitly, using a collar $S^3 \times [0,1]$ about $\partial B^4 = S^3 \times \{0\}$, let \[S' = \partial S \times [0,1/2] \cup \textrm{int}(S)\times \{1/2\}\] and $\gamma \subset \textrm{int}(S) \subset S^3 \times \{1/2\}$ be a simple closed curve. Applying Proposition \ref{HiroseHopfBand} yields

 \begin{lemma}
 Suppose that the annulus $A =\overline{\nu_{\gamma \subset S}}$ is a Hopf band in $S^3 \times \{1/2\}$. Let $\psi$ be an isotopy of $B^4$ which pushes the interior of $A$ futher into $\textrm{int}(B^4)$. There exists a diffeomorphism $T: B^4 \to B^4$ for which  $\psi^{-1} \circ T \circ \psi : B^4 \to B^4$ is the identity on the boundary $S^3 \times \{0\}$, $T(S') = S'$ and induces the Dehn twist along a core curve of the annulus $A$. In particular, $T_{\gamma} \in \mathcal{E}_{\partial}(B^4,S').$
\end{lemma}

A straightforward generalization of these methods shows that the monodromy of any \textit{fibered subsurface} can be realized as an extendable mapping class.  \newline

\textbf{Twists along on separating unlinks.} In the sequel we make use of an argument in \cite[Lemma 2.2]{HiroseS^4}. For the benefit of the reader we include a reformulation (see also \cite{LawandeSaha} formulating this in the more general setting of arbitrary open book decompositions). 

\begin{figure}
  \centering
  \def\svgwidth{0.5\columnwidth}
  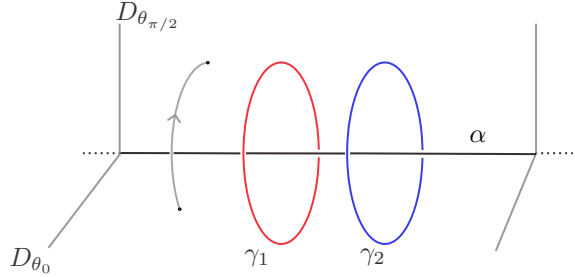
  \caption{\textnormal{Depicted is an axis $\alpha \subset S^3$ for a 2-component unlink $\gamma_1\cup \gamma_2$. The unlink is transverse to the pages of the open book decomposition $\pi : S^3\backslash \alpha \to S^1$. } }
  \label{shiftingfibers}
\end{figure}

\begin{proposition}[cf. Lemma 2.2 in \cite{HiroseS^4}] \label{separatingunlink}
If a surface $\Sigma \subset M^4$ intersects an $S^3 \times [0,1] \subset \mathrm{int}(M^4)$ in the cylinder on an $n$ component unlink $\gamma_1,\ldots,\gamma_n \subset S^3$, the multitwist $T_{\gamma_1}\dots T_{\gamma_n} \in \mathcal{E}_{\partial}(M,\Sigma)$.
\end{proposition}

\begin{proof}
Let $\alpha \subset S^3$ be an axis for the unlink $\gamma_1\cup \cdots \cup \gamma_n$ (see Figure \ref{shiftingfibers}). Consider the standard open book decomposition $\pi :S^3 \backslash \alpha \cong \textrm{int}(D^2)\times S^1 \to S^1$. Pick $n$ distinct points $\{x_1,\ldots,x_n\} \subset \textrm{int}(D^2)$. We may assume the unlink $\gamma_1\cup \cdots \cup \gamma_n$ is of the form $\{x_1,\ldots,x_n\} \times S^1$. By isotopy extension, the monodromy of $\pi$ yields an isotopy \[\varphi : S^3 \times [0,1]\to S^3\] with $\varphi\rvert_{S^3\times \{0\}}= \varphi\rvert_{S^3\times \{1\}}= \textrm{Id}_{S^3}$. The isotopy $\varphi$ spins the points $\{x_1,\ldots,x_n\}$ around the axis $\alpha$ whilst fixing the unlink $\gamma_1 \cup \cdots \cup \gamma_n$ setwise. Define $\Phi: S^3 \times [0,1] \to S^3 \times [0,1]$ by $\Phi(x,t) = (\varphi(x,t),t)$. The map $\Phi$ induces a Dehn twist on each of the annuli $\gamma_1 \times [0,1] \cup \cdots \cup \gamma_n \times [0,1]$. Extending $\Phi$ by the identity over $M$ shows $T_{\gamma_1}\cdots T_{\gamma_n} \in \mathcal{E}_{\partial}(M,\Sigma)$.
\end{proof}

When applying Proposition \ref{separatingunlink}, we suppress conjugation by an isotopy arranging the candidate surface so that it intersects an $S^3\times [0,1]$ in the cylinder on an unlink. \newline

Hirose \cite{HiroseS^4} proved that unknotted surfaces in $S^4$ are spin flexible (the genus one case is due to Montesinos \cite{Montesinos}). In more detail, Hirose constructed generating sets for the stabilizer of Rochlin's form (an \textit{even} quadratic form), using the methods above to exhibit each generator as an extendable mapping class. 

We conclude this section by observing that spin-flexibility for unknots extends in a straightforward way to the \textit{rel boundary} setting. This result is used to lift mapping classes along an unknotted subsurface in Lemma \ref{twistson1curves}. 

\begin{proposition} \label{relboundaryhirose}
For $g \geq 1$, let $\Sigma_g^1 \subset B^4$ denote an unknotted genus $g$ Seifert surface for an unknot in $S^3 = \partial B^4$ with its interior pushed into the interior of $B^4$. Then $\mathcal{E}_{\partial}(B^4,\Sigma_g^1) = \Mod(\Sigma_g^1)[q_{\Sigma_g^1}]$.  \end{proposition}

\begin{proof}
Cap off the boundary component $\Delta$ of $\Sigma_g^1$ with a closed embedded disk in $S^3$. Let $\Sigma_g$ denote the resulting surface. Rochlin's form $q := q_{\Sigma_g^1}$ extends over $\Sigma_g$ (inclusion induces an isomorphism $H_1(\Sigma_g^1;\mathbb{Z}/2\mathbb{Z}) \cong H_1(\Sigma_g;\mathbb{Z}/2\mathbb{Z})$). Therefore, by (the restriction of) \textit{disk-pushing} \cite{Primer}, there is a short exact sequence 
$$ 1 \to \pi_1(UT\Sigma_g) \to \Mod(\Sigma_g^1)[q] \to \Mod(\Sigma_g)[q] \to 1,$$
where $UT\Sigma_g \to \Sigma_g$ denotes the unit tangent bundle of $\Sigma_g$. Hirose \cite{HiroseS^4} constructs, for $g\geq 2$, a generating set for $\Mod(\Sigma_g)[q]$ consisting of multitwists supported in $\Sigma_g^1$. The methods in \cite{HiroseS^4} show each element has a lift to $\mathcal{E}_{\partial}(B^4,\Sigma_g^1)$. The same is true for $g = 1$ (see \cite{HiroseSurvey}).

It now suffices to prove that \[\pi_1(UT\Sigma_g) \leq \mathcal{E}_{\partial}(B^4,\Sigma_g^1).\]   If  $\widetilde{\gamma} \in \pi_1(UT\Sigma_g)$ has image $\gamma \in \pi_1(\Sigma_g)$ (based at point in $\Delta$) under projection $UT\Sigma_g\to \Sigma_g$, the image of $\widetilde{\gamma}$ in $\Mod(\Sigma_g^1)$ is the multitwist $T_{\gamma_L}T_{\gamma_R}^{-1}T_{\Delta}^k$ for some $k \in \mathbb{Z}$, where $\gamma_L$ (resp. $\gamma_R$) is obtained by pushing $\gamma$ to the left (resp. right) over the disk $\Delta$. The curves $\gamma_L \cup \gamma_R \cup \Delta$ bound a pair of pants on $\Sigma_g^1$.

Let $\{a_k,b_k\}_{k=1}^{g}$ denote the standard generating set for $\pi_1(\Sigma_g,*)$, $* \in \partial \Sigma_g^1$. In Figure \ref{pantsrelbdry}, $a_{k,L} \cup a_{k,R} \cup \Delta$ and $b_{k,L} \cup b_{k,R} \cup \Delta$ are 3-component separating unlinks in an $S^3$ slice of a collar about the boundary. By Proposition \ref{separatingunlink}, the multitwists along the boundary components of the pair of pants are in $\mathcal{E}_{\partial}(B^4,\Sigma_g^1)$.

By the tube isotopy (Figure \ref{squaretwist}), the square of each Dehn twist along $a_k,b_k$ is in $\mathcal{E}_{\partial}(B^4,\Sigma_g^1)$. By Proposition \ref{separatingunlink}, we have $T_\Delta \in \mathcal{E}_{\partial}(B^4,\Sigma_g^1)$. As disjoint twists commute, it follows that a generating set for $\pi_1(UT\Sigma_g)$ is contained in $\mathcal{E}_{\partial}(B^4,\Sigma_g^1)$. 
\end{proof}

\begin{figure}
  \centering
  \def\svgwidth{0.85\columnwidth}
  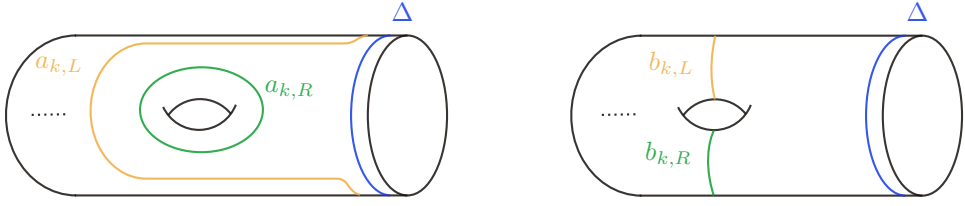
  \caption{\textnormal{The union $a_{k,L}\cup a_{k,R} \cup \Delta$ forms a 3-component unlink in $S^3$, similarly for $b_{k,L} \cup b_{k,R} \cup \Delta$}.}
  \label{pantsrelbdry}
\end{figure}

\section{Framed and Spin mapping class groups}

\subsection{Framings} Recall that a \textit{framing} of a manifold is a trivialization of its tangent bundle. Let $\Sigma = \Sigma_g^b$ be a compact oriented surface of genus $g$ with $b \geq 1$ boundary components. Fix a framing $\phi : T\Sigma \cong \Sigma \times \mathbb{R}^2$. Let $\mathcal{S}(\Sigma)$ denote the set of oriented isotopy classes of simple closed curves on $\Sigma$. The framing $\phi$ determines a \textit{winding number function} (again denoted) $\phi : \mathcal{S}(\Sigma) \to \mathbb{Z}$  by \[\phi(\gamma) = \deg(\phi \circ \overrightarrow{\gamma} : S^1 \to \mathbb{R}^2 -\{0\}) \in \mathbb{Z}.\]  Here $\overrightarrow{\gamma}$ is the forward pointing velocity vector field of $\gamma$. Winding number functions depend only on the isotopy class of framing. We abuse terminology, identifying winding number functions with their respective framings. As $\Sigma$ is oriented, we frequently specify an isotopy class of framing by means of a nowhere vanishing vector field. Winding number functions satisfy the following basic properties as established by Humphries-Johnson in \cite{HJ}.

\begin{proposition}
\label{twistlinear}
For a framed surface $(\Sigma,\phi)$ the function $\phi : \mathcal{S}(\Sigma) \to \mathbb{Z}$ satisifes  
\begin{enumerate}
    \item (Twist-linearity) For any $a,b \in \mathcal{S}(\Sigma)$, 
    $$ \phi(T_a(b)) = \phi(b) + \langle b,a\rangle \phi(a)$$
    where $\langle -,-\rangle : H_1(\Sigma;\mathbb{Z}) \times H_1(\Sigma;\mathbb{Z}) \to \mathbb{Z}$ denotes the oriented intersection pairing. 
    \item (Homological coherence) Let $S \subset \Sigma$ be a subsurface with boundary components $c_1,\ldots,c_k$ oriented with $S$ to their left. Then 
    $$\sum_{i=1}^k \phi(c_k) = \chi(S)$$
    where $\chi(S)$ denotes the Euler characteristic of $S$. 
\end{enumerate}

\end{proposition}

A framing of a surface may be extended across a stabilization of its boundary. Precisely, we have 

\begin{lemma}[Lemma 3.20 in \cite{IshanNick}]
Let $(\Sigma,\phi)$ be a framed surface. Let $\Sigma^+$ be a surface obtained from $\Sigma$ by attaching a 1-handle along $\partial \Sigma$. Fix $w \in \mathbb{Z}$. Let $c \subset \Sigma^+$ be a curve for which $c \cap \Sigma$ is a single arc. Then there is a unique isotopy class of framing $\phi^+$ of $\Sigma^+$, extending $\phi$, for which $\phi^+(c) = w \in \mathbb{Z}$.
\label{extendframing}
\end{lemma}

\textbf{Framed mapping class groups.} Let $\phi : \mathcal{S}(\Sigma) \to \mathbb{Z}$ be a framing of $\Sigma$. Let $\Mod(\Sigma)$ act on the set of (isotopy classes of) framings by pullback on the determined winding number functions. Precisely, we set $(f \cdot \phi)(x) = \phi(f^{-1}(x))$ for $x \in \mathcal{S}(\Sigma)$. 

\begin{defn} \label{framedmcg}
Let $(\Sigma,\phi)$ be a framed surface. The stabilizer of $\phi$ is denoted 
\[\Mod(\Sigma)[\phi] = \{f \in \Mod(\Sigma) \,|\, (f\cdot \phi) = \phi\}.\] We call $\Mod(\Sigma)[\phi]$ a \textit{framed mapping class group}.
\end{defn}

Following Calderon-Salter \cite{CalderonSalter}, we list a variety of foundational results regarding the structural properties of framed mapping class groups to be used in the sequel.  

\begin{defn}
Let $(\Sigma,\phi)$ be a framed surface. A nonseparating simple closed curve $\gamma \subset \Sigma$ is called \textit{admissible} if $\phi(\gamma) = 0$. 
\end{defn}

 By twist linearity, if $\gamma \subset \Sigma$ is admissible,  $T_{\gamma} \in \Mod(\Sigma)[\phi]$. The  \textit{admissible subgroup} is defined as \[\mathcal{T}_{\phi} := \langle T_{\gamma} \;|\; \gamma  \; \mathrm{ admissible}\rangle \leq \Mod(\Sigma)[\phi].\] A foundational result is

\begin{theorem}[Proposition 5.11 in \cite{Strata3}] \label{AdmissibleTwists}
For a framed surface $(\Sigma_g^b,\phi)$, where $g \geq 5$ and $b\geq 1$, we have  
\[\Mod(\Sigma_g^b)[\phi] = \mathcal{T}_{\phi}.\]
\end{theorem}

We record the following instances of the \textit{framed} change-of-coordinates principle (a corollary of \cite[Proposition 2.15]{Strata3}). 

\begin{lemma}
Let $(\Sigma,\phi)$ be a framed surface. Let $\alpha$ and $\beta$ be two nonseparating simple closed curves with $\phi(\alpha) = \phi(\beta)$. Then there exists $f \in \Mod(\Sigma)[\phi]$ for which $f(\alpha) = \beta$. 
\end{lemma}

\begin{lemma} \label{dualcurve}
Let $(\Sigma,\phi)$ be a framed surface. Let $\alpha \subset \Sigma$ be a nonseparating simple closed curve. Then there exists an admissible curve $\beta \subset \Sigma$ which is geometrically dual to $\alpha$, that is $i(\alpha,\beta) = |\alpha \cap \beta| = 1$. 
\end{lemma}

A result that plays a fundamental role in this paper is a stabilization lemma for framed mapping class groups. 
 \begin{lemma}[Lemma 5.13 in \cite{Strata3}]\label{stabilization}
 Let $(\Sigma,\phi)$ be a framed surface. Let $S \subset \Sigma$ be a subsurface of genus at least 5 and $a \subset \Sigma$ an admissible curve such that $a \cap S$ is a single arc. Let $S^+$ be a regular neighborhood of $S\cup a$. Then 
 $$ \Mod(S^+)[\phi] = \langle \Mod(S)[\phi],T_a\rangle.$$
 \end{lemma}

\subsection{$\mathbb{Z}/r\mathbb{Z}$ winding number functions.} Recall that $\mathcal{S}(\Sigma)$ denotes the set of oriented isotopy classes of simple closed curves on $\Sigma$.

\begin{defn}
A $\mathbb{Z}/r\mathbb{Z}$ \textit{winding number function} $\psi$ on $\Sigma$ is a function \[\psi : \mathcal{S}(\Sigma) \to \mathbb{Z}/r\mathbb{Z}\] satisfying twist-linearity and homological coherence modulo $r$ (cf. Proposition \ref{twistlinear}).
\end{defn}

\begin{rmk}
As a consequence of homological coherence, winding number functions satisfy \textit{reversibility}. For $c \in \mathcal{S}(\Sigma)$, let $\overline{c} \in \mathcal{S}(\Sigma)$ denote $c$ with the opposite orientation. Then $\psi(\overline{c}) = -\psi(c) \in \mathbb{Z}/r\mathbb{Z}$. 
\end{rmk}

\begin{eg} \label{reduceframing} A $\mathbb{Z}/r\mathbb{Z}$ winding number function may be constructed as follows \cite[Definition 2.6]{Strata3}. Consider a framed surface $(\Sigma,\phi)$. Let $\overline{\Sigma}$ denote the surface obtained by capping off boundary components $\Delta_1,\ldots,\Delta_n$ of $\Sigma$ with closed disks (note $\overline{\Sigma}$ may still have boundary). Orient each component $\Delta_k$ so that $\Sigma$ is to the left. Isotoping a curve on $\Sigma$ across the disk bounding $\Delta_k$ changes the $\phi$-framing by $\phi(\Delta_k)+1$ (homological coherence). Let $r = \gcd(\phi(\Delta_1)+1,\ldots,\phi(\Delta_n)+1)$ and consider $\overline{\phi} : \mathcal{S}(\overline{\Sigma}) \to \mathbb{Z}/r\mathbb{Z}$ defined by \[\overline{\phi}(\gamma) = \phi(\gamma) \mod r.\] We say the winding number function $\overline{\phi}$ is obtained by \textit{reducing} $\phi$ mod $r$.  
\end{eg}

\begin{defn}
A nonseparating simple closed curve $\gamma \subset 
\Sigma$ is \textit{admissible} under $\psi : \mathcal{S}(\Sigma) \to \mathbb{Z}/r\mathbb{Z}$ if $\psi(\gamma) = 0 \in \mathbb{Z}/r\mathbb{Z}$. 
\end{defn}

As in the framed settting, we define the admissible subgroup \[\mathcal{T}_{\psi} = \langle T_{\gamma} \,|\, \gamma \subset \Sigma\;\; \psi-\textrm{admissible}\rangle.\] The set of $\mathbb{Z}/r\mathbb{Z}$ winding number functions carries an action of the mapping class group via pullback. 
\begin{defn} \label{rspinmcg}
The stabilizer of a $\mathbb{Z}/r\mathbb{Z}$ winding number function $\psi$ on $\Sigma$ is denoted
\[\Mod(\Sigma)[\psi] = \{ f\in \Mod(\Sigma) \,|\, (f\cdot \psi) = \psi\}.\]
\end{defn}

\begin{rmk}
The stabilizer of a $\mathbb{Z}/r\mathbb{Z}$ winding number function is called an \textit{r-spin mapping class group}. The terminology is motivated by a bijective correspondence between $\mathbb{Z}/r\mathbb{Z}$ winding number functions and \textit{r-spin structures} on $\Sigma$ \cite[Proposition 3.6]{IshanNick}. Let $UT\Sigma \to \Sigma$ be the unit tangent bundle of $\Sigma$.\footnote{Relative to an auxillary Riemannian metric.} An $r$-\textit{spin structure} on $\Sigma$ is a cohomology class in $H^1(UT\Sigma;\mathbb{Z}/r\mathbb{Z})$ which assigns the value $1\in \mathbb{Z}/r\mathbb{Z}$ to the class of a positively oriented fiber.
\end{rmk}

For closed surfaces we have an analogue of Theorem \ref{AdmissibleTwists}.

\begin{theorem}[Proposition 5.1 in \cite{CalderonSalter}]
Assume that $\Sigma$ is closed and has genus $g(\Sigma) \geq 3$. Let $\psi$ be a $\mathbb{Z}/r\mathbb{Z}$ winding number function on $\Sigma$. Then \[\Mod(\Sigma)[\psi] = \mathcal{T}_{\psi}.\]
\label{admissiblespin}
\end{theorem}

\textbf{Quadratic forms and $\mathbb{Z}/2\mathbb{Z}$ winding number functions.} 
In the sequel we require passage between quadratic forms and winding number functions on surfaces with boundary. Following \cite{Johnsonspin} and Section 2.3 of \cite{Randal-Williams}, a calculation gives
\begin{proposition} \label{QuadraticFormSpin}
A winding number function $\psi : \mathcal{S}(\Sigma) \to \mathbb{Z}/2\mathbb{Z}$ defines a quadratic form $q_{\psi} : H_1(\Sigma;\mathbb{Z}/2\mathbb{Z}) \to \mathbb{Z}/2\mathbb{Z}$ by the requirement that 
\[q_{\psi}(\gamma) = \psi(\gamma) + 1\mod 2\]
for $\gamma \in \mathcal{S}(\Sigma)$. Conversely, a quadratic form on $H_1(\Sigma;\mathbb{Z}/2\mathbb{Z})$ defines a $\mathbb{Z}/2\mathbb{Z}$ winding number function by the same formula. Moreover, quadratic forms on $H_1(\Sigma;\mathbb{Z}/2\mathbb{Z})$ are determined by their values on a generating set for $H_1(\Sigma;\mathbb{Z}/2\mathbb{Z})$.  
\end{proposition}

\begin{rmk} \label{admissibleunderform}
In the presence of a quadratic form $q : H_1(\Sigma;\mathbb{Z}/2\mathbb{Z}) \to \mathbb{Z}/2\mathbb{Z}$, a nonseparating simple closed curve $\gamma \subset \Sigma$ is called \textit{admissible} if $q(\gamma) = 1 \mod 2$. 
\end{rmk}

\section{(Spin) Flexible Surfaces in $\mathbb{C}P^2$ and $S^2 \times S^2$}

We begin by considering a family of framed Seifert surfaces, following \cite{Ferretti}. After stabilizing with an unknotted surface of sufficiently large genus, we use the inductive techniques recalled in Section 3 to prove Lemma \ref{twistson1curves}. As a consequence we obtain Theorem \ref{thmA}. \newline

\textbf{Seifert surfaces for homogeneous braid closures}. Recall that a braid $\beta$ on $n$ strands can be expressed as a word in $n$ generators $\sigma_1,\ldots,\sigma_{n-1}$, where the letter $\sigma_i$ consists of a single over-crossing between the $(i+1)$-st and $i$-th strands. 

\begin{defn}[\cite{Stallings}]
A word $\beta = \sigma_{\iota_1}^{\varepsilon_1}\cdots \sigma_{\iota_k}^{\varepsilon_k}$, $\varepsilon_i = \pm 1$, is \textit{homogeneous} if every letter $\sigma_i$ occurs at least once, and the exponents of all occurences of a given letter $\sigma_i$ are the same.
\end{defn}

Following Stallings \cite{Stallings}, a Seifert surface $F_\beta$ is constructed for the closure of $\beta$ using a stack of $n$ disks, joining adjacent disks with half-twisted bands. See Figure \ref{FramedSurface}. We prefer to use half curls in our bands. Our conventions are that positive letters correspond to down rightward curls and negative letters correspond to down leftward curls. 

\begin{rmk}
For a homogeneous braid word $\beta$, $F_\beta$ is a \textit{fibered surface} in $S^3$. In fact, $F_\beta$ is obtained by \textit{plumbing} a sequence of positive/negative Hopf bands (see \cite{Stallings} for details).  
\end{rmk}

Any adjacent pair of bands attached to the same pair of disks determines a simple closed curve on $F_\beta$ by connecting the cores of the bands with simple arcs in the disks. We write $\mathcal{C}_{F_\beta}$ for the finite collection of all such (isotopy classes of) simple closed curves. See Figure \ref{FramedSurface}. We call $\mathcal{C}_{F_\beta}$ the \textit{braid word assemblage} on $F_\beta$. 

\begin{defn}[\cite{Ferretti}]
Define \[\mathcal{T}_{\mathcal{C}_{F_\beta}} = \langle T_{\gamma} \,|\, \gamma \in \mathcal{C}_{F_\beta}\rangle.\] We call $\mathcal{T}_{\mathcal{C}_{F_\beta}}$ the \textit{geometric monodromy group} of the braid word $\beta$.  
\end{defn}

Orient curves in $\mathcal{C}_{F_\beta}$ as in Figure \ref{FramedSurface}. As $\beta$ is homogeneous, each curve in $\mathcal{C}_{F_\beta}$ is a Hopf band curve (Definition \ref{Hopfbandcurve}). By construction, the surface $F_\beta$ deformation retracts onto the union of the curves in $\mathcal{C}_{F_\beta}$. \newline

\textbf{Framing the surface $F_\beta$}. Following \cite{Ferretti}, we frame the surface $F_\beta \subset S^3$ through the nowhere vanishing vector field depicted in Figure \ref{FramedSurface}.  Let \[\phi_\beta : \mathcal{S}(F_\beta) \to \mathbb{Z}\] be the associated winding number function on the set of oriented isotopy classes of simple closed curves on $F_\beta$. The constructed vector field is parallel to the core of the 1-handles and so contribution to the winding number $\phi_\beta$ takes place only in the 0-handles. Consider an oriented simple closed curve $\gamma \subset F_\beta$. Let $a_R^{\gamma}$ (resp. $a_L^{\gamma}$) denote the number of `rightward' (resp. `leftward') arcs of $\gamma$ in the 0-handles of $F_{\beta}$. We compute  \[\phi_\beta(\gamma) = \frac{1}{2}(a_L^{\gamma}-a_R^{\gamma}).\] Any curve in $\mathcal{C}_{F_\beta}$ is therefore admissible under the framing and there is the sequence $$\mathcal{T}_{\mathcal{C}_{F_\beta}} \leq \mathcal{T}_{\phi_\beta} \leq \Mod(F_\beta)[\phi_\beta].$$ Ferretti proves \cite[Theorem 2]{Ferretti} that for a sufficiently general collection of positive braids\footnote{All but finitely many prime positive braids not of type $A_n$.} $\beta$ whose closure is a knot, the above is a sequence of equalities, namely  \[\mathcal{T}_{\mathcal{C}_{F_\beta}} = \Mod(F_\beta)[\phi_\beta].\] 

\begin{corollary} \label{framedlowerbound}
Let $F_\beta'$ denote the result of pushing the interior of $F_\beta$ into the interior of $B^4$. Pull back the framing $\phi_\beta$ to a framing $\phi_\beta'$ of $F_\beta'$. Then, for a braid $\beta$ to which Theorem 2 of \cite{Ferretti} applies, we have \[\Mod(F_\beta')[\phi_\beta'] \leq \mathcal{E}_\partial(B^4,F_\beta').\] 
\end{corollary}

We use Corollary \ref{framedlowerbound} in proving Theorem \ref{ThmB}.

\begin{figure}
    \centering
    \includegraphics[width=0.65\linewidth]{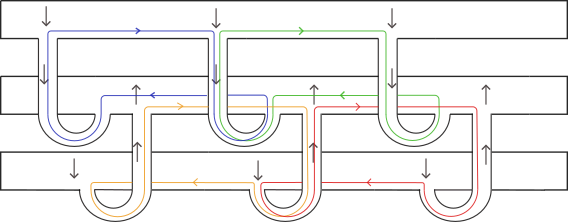}
    \caption{\textnormal{The framed Seifert surface $F_\beta$ associated to the word $\beta = (\sigma_1\sigma_2^{-1})^3$ together with its braid word assemblage of Hopf band curves.}}
    \label{FramedSurface}
\end{figure}

\subsection{An integral lift of the linking form on a stabilization of $F_\beta$.} \label{integralliftsection}{ Let $\Sigma_g \subset S^3 \backslash F_\beta$ be the boundary of a genus $g$ handlebody. Consider a standard internal connect sum of $F_\beta$ with $\Sigma_g$ along a smoothly embedded disk $D_0 \subset F_{\beta}$. We assume the disk $D_0$ sits inside a 0-handle of $F_\beta$ which has at least a pair of bands attached along its boundary (if no such 0-handle exists the surface $F_\beta$ is a disk). Denote the resulting surface \[S_{\beta,g} = F_\beta \# \Sigma_g \subset S^3.\] Let $\widehat{F}_\beta = F_\beta - \textrm{int}(D_0) \subset S_{\beta,g}$ and $\widehat{\Sigma}_g = \Sigma_g - \textrm{int}(D_0) \subset S_{\beta,g}$. The framing $\phi_{\beta}$ of $F_{\beta}$ restricts to the subsurface $\widehat{F}_\beta$, denoted again by $\phi_{\beta}$. This framing does not extend over the surface $S_{\beta,g}$. Indeed, winding number functions satisfy homological coherence. On $F_{\beta}$ the curve $\partial D_0$ bounds a disk, whilst on $S_{\beta,g}$ it bounds a surface of genus $g$.

We circumvent this obstruction by removing a small disk $D_1$ in the interior of the unknotted piece $\widehat{\Sigma}_g$. Write  \[P_{\beta,g} = S_{\beta,g} - \textrm{int}(D_1).\] The framing $\phi_\beta$ now extends, by Lemma \ref{extendframing}, to a framing $\phi_\beta^+$ of $P_{\beta,g}$ using the chain of curves $\mathcal{U}_g = \{c_1,\ldots,c_{2g-1}\} \subset \widehat{\Sigma}_g$ and the curve $\mu \subset S_{\beta,g}$ depicted in Figure \ref{stabilizingseifertsurface}, setting $\phi_\beta^+(\mu) = 0$ and $\phi_\beta^+(c_k) = 1$ for each $k \in \{1,\ldots,2g-1\}$. The curve $\mu$ is a homologically the sum of a curve $\gamma \subset \mathcal{C}_{F_\beta}$ and a meridian in $\widehat{\Sigma}_g$. The essential property of this extension is that the resulting framing is an integral lift of the linking form on the Seifert surface $S_{\beta,g}$. We now consider the winding number function induced by the framing $\phi^+_\beta$ on the surface $S_{\beta,g}$ (as in Example \ref{reduceframing}).  \label{SPFDefns}}

\begin{figure}
  \centering
  \def\svgwidth{0.85\columnwidth}
  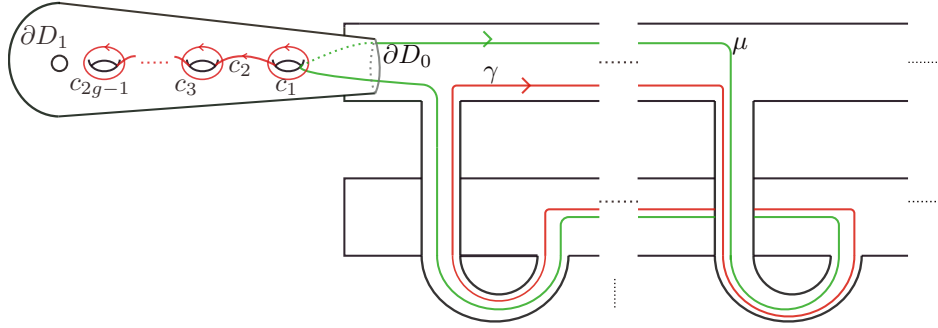
  \caption{\textnormal{We stabilize $F_\beta$ along the chain of curves $\{\mu,c_1,\ldots,c_{2g-1}\}$}. }
  \label{stabilizingseifertsurface}
\end{figure}

Henceforth suppress notational dependency on $\beta$ and $g$, writing $S = S_{\beta,g}$, $P = P_{\beta,g}$ and $F = F_{\beta}$.  The framing $\phi^+$ induces a winding number function \[\psi  : \mathcal{S}(S) \to \mathbb{Z}/2g\mathbb{Z}\] on $S$. To see this, orient the curves $\partial D_0$ and $\partial D_1$ so that $\widehat{\Sigma}_g$ is to the left. By homological coherence, \[\phi^+(\partial D_1) + \phi^+(\partial D_0)  = 2-2g - 2 = -2g\] and so $\phi^+(\partial D_1) +1 = -2g$. Let $\overline{\psi} : \mathcal{S}(S) \to \mathbb{Z}/2\mathbb{Z}$ denote the mod 2 reduction of $\psi$; in particular, $\overline{\psi} = \phi^+ \mod 2$ for curves in $P$. Recall that $S = F \# \Sigma_g\subset S^3$ is a Seifert surface for the link $\partial F$, $\theta_S : H_1(S;\mathbb{Z}/2\mathbb{Z}) \to \mathbb{Z}/2\mathbb{Z}$ denotes the mod 2 linking form on $S$, and that  $\overline{\psi} +1 : H_1(S;\mathbb{Z}) \to \mathbb{Z}/2\mathbb{Z}$ defines a quadratic form   (Proposition \ref{QuadraticFormSpin}).

\begin{proposition} \label{framingmod2}
The  quadratic form $\overline{\psi}+1$ agrees with the linking form $\theta_S$ on $H_1(S;\mathbb{Z}/2\mathbb{Z})$. Precisely  \[\overline{\psi}(\gamma) +1= \theta_S(\gamma)   \mod 2\]for a simple closed curve $\gamma \subset S$.  
\end{proposition}

\begin{proof}
By Proposition \ref{QuadraticFormSpin}, it suffices to prove that the quadratic forms $\overline{\psi}+1$ and $\theta_S$ agree on a generating set for $H_1(S;\mathbb{Z}/2\mathbb{Z})$. The collection $\mathcal{C}_{F} \cup \{\mu\} \cup \mathcal{U}_g$ is a basis for $H_1(S;\mathbb{Z}/2\mathbb{Z})$. Each curve in $\mathcal{C}_{F}$ is a Hopf band curve and has $\phi^+$-framing zero. The curve $\mu$ is a Hopf band curve on the Seifert surface $S$ and has $\phi^+$-framing zero. Each curve in $\mathcal{U}_g$ has self-linking number zero and $\phi^+$-framing one.   
\end{proof}

\textbf{Extending admissible twists rel boundary.} Let $S' \subset B^4$ be obtained by pushing the interior of $S$ into the interior of $B^4$. For lifting diffeomorphisms along subsurfaces and extending by the identity, an explicit model is convenient. Let $S^3 \times [0,1]$ be a collar about $\partial B^4 = S^3 \times \{0\}$, and set $$S' = \partial F \times [0,1/4] \cup \widehat{F} \times \{1/4\} \cup \partial D_0\times [1/4,1/2] \cup \widehat{\Sigma}_g\times \{1/2\}.$$ We suppress identification between $S$ and $S'$ and henceforth pull back all data on $S$ to the surface $S'$. A subsurface $N$ of $S$ pulled back to $S'$ will be denoted $N'$, the same rule applying to curves and configurations thereof. Recall that the linking form on $H_1(S;\mathbb{Z}/2\mathbb{Z})$ pulls back to Rochlin's form $q_{S'}: H_1(S';\mathbb{Z}/2\mathbb{Z}) \to \mathbb{Z}/2\mathbb{Z}$ (Proposition \ref{pullbackRochlin}). We now exploit the unknotted piece of the surface $S'$, in conjunction with the inductive framework behind the theory of framed mapping class groups,  to manufacture more extendable mapping classes in $\Mod(S')$.

\begin{lemma} \label{twistson1curves}
Fix $g \geq 5$ and a homogeneous braid word $\beta$. Let $\gamma \subset S_{\beta,g}'$ be a nonseparating simple closed curve  with $q_{S_{\beta,g}'}(\gamma) = 1$. Then  \[T_{\gamma} \in \mathcal{E}_{\partial}(B^4,S_{\beta,g}').\]    
\end{lemma}

\begin{proof}
Recall that we suppress subscripts indicating dependency on $\beta$ and $g$. By an isotopy if necessary, we may assume that $\gamma \subset P'$. By Propositions \ref{framingmod2} and \ref{pullbackRochlin}, \[\phi^+  = q_{S'}+1 \mod 2\] on $P'$. After orienting $\gamma$, $\phi^+(\gamma) = 2k$ for some $k \in \mathbb{Z}$. Let $p_{D_1'} : \Mod(P') \to \Mod(S')$ be the homomorphism induced by extension by identity over the closed 2-disk $D_1'$. We claim that \[p_{D_1'}(\Mod(P')[\phi^+]) \leq \mathcal{E}_{\partial}(B^4,S').\] Let $V' = P' \cap \widehat{\Sigma}_g \times \{1/2\} \subset S^3 \times \{1/2\}$. The surface $V'$ is our genus $g$ unknot with two standard boundary components. The surface $P'$ is obtained by stabilizing $V'$ along the $\phi^+$-admissible configuration $\mathcal{Z}' = \{\mu'\} \cup \mathcal{C}_{F}'\backslash\{\gamma'\}$. As the genus of $V'$ is at least 5, it follows from Lemma \ref{stabilization} that \[\Mod(P')[\phi^+] = \langle \Mod(V')[\phi^+], \,T_c \,:\, c \in \mathcal{Z}'\rangle.\] Each curve $c \in \mathcal{Z}'$ is (after an isotopy of $S'$ rel $\partial S'$) a Hopf band curve in an $S^3$ slice of the collar and so $T_c \in \mathcal{E}_\partial(B^4,S')$. By Theorem \ref{AdmissibleTwists}, $\Mod(V')[\phi^+]$ is generated by Dehn twists on admissible curves. Fix an admissible curve $\zeta \subset V'$. By Proposition \ref{framingmod2}, it follows that \[q_{S'}(\zeta) = \mathrm{lk}(\zeta,\zeta^+) = 1 \mod 2,\]the linking number computed in $S^3 \times \{1/2\}$. By Proposition \ref{relboundaryhirose}, $T_{\zeta} \in \mathcal{E}_{\partial}(B^4,S')$.

Fix a nonseparating simple closed curve $\delta \subset \widehat{\Sigma}_g \times \{1/2\} \subset S'$ with $\phi^+(\delta)  = 2k$. By the framed change of coordinates principle (Proposition \ref{framedchange}), there exists $f \in \Mod(P')[\phi^+]$ such that $f(\gamma) = \delta$. Then \[T_{\delta} = fT_{\gamma}f^{-1},\] and by spin-flexibility of $\widehat{\Sigma}_g\times\{1/2\}$ rel boundary (Proposition \ref{relboundaryhirose}), we have $T_{\delta} \in \mathcal{E}_{\partial}(B^4,S')$ and  $T_{\gamma} \in \mathcal{E}_{\partial}(B^4,S')$. 
\end{proof}

We now construct (spin) flexible representatives in $\mathbb{C}P^2$ and $S^2 \times S^2$. 

\subsection{Surfaces in $\mathbb{C}P^2$ and $S^2 \times S^2$} \label{ConstructingSurfaces} Work with the standard handlebody descriptions for the manifolds $\mathbb{C}P^2$ and $S^2 \times S^2$ - after removing a 4-ball in each, $\mathbb{C}P^2$ is obtained by attaching a 2-handle along a $+1$-framed unknot in $\partial B^4$, and $S^2\times S^2$ by attaching a pair of 2-handles along a $0$-framed Hopf link in $\partial B^4$.

Let $\mathbb{C}P^1 \subset \mathbb{C}P^2$ be a line. Recall that a representative for the class $d[\mathbb{C}P^1] \in H_2(\mathbb{C}P^2;\mathbb{Z})$ may be constructed as follows (see \cite{GS} for a thorough treatment). Let $B^4 \subset \mathbb{C}P^2$ be the 0-handle. In $\partial B^4$ take $d$ parallel copies of the attaching circle for the 2-handle, an oriented $(d,d)$ torus link, which we represent as the closure of the  braid $\beta_d = (\sigma_1\cdots\sigma_{d-1})^d \in B_d$. We cap off the pushed in Seifert surface $S_{\beta_d,g}'$ for $\beta_d$ constructed in Subsection \ref{SPFDefns} with $d$ disjoint 2-disks, which are parallel copies of the core of the 2-handle. This yields a closed oriented surface $Z_d = Z_{d,g} \subset \mathbb{C}P^2$ with $[Z_d] = d[\mathbb{C}P^1] \in H_2(\mathbb{C}P^2;\mathbb{Z})$ and genus $g(Z_d) = \frac{1}{2}(|d|-1)(|d|-2)+g.$ \newline

Fix $x \in S^2$. Consider \[(n,m) = n\left [\{x\}\times S^2 \right ]+ m\left [S^2 \times \{x\}\right]\in H_2(S^2 \times S^2;\mathbb{Z}),\] the class of $n$ vertical and $m$ horizontal spheres. We call a class $(n,m)$ \textit{positive} if both coefficients are strictly positive, and \textit{negative} if one of $n$ or $m$ is strictly negative. Note that, if either of $n$ or $m$ is zero, the class admits a genus zero representative, obtained by tubing together a disjoint collection of 2-spheres. 

We start by constructing representatives for positive classes. Fix $n,m >0$. In the boundary of the 0-handle, take $n$ parallel copies of one (vertical) attaching circle and $m$ parallel copies of the other (horizontal), the closure of a positive braid  $\beta_{n,m}$ on $n+m$ strands. Let $Z_{n,m,g} = Z_{n,m} \subset S^2 \times S^2$ be the closed surface obtained by capping off $S_{\beta_{n,m},g}'$ (Subsection \ref{SPFDefns}) with parallel copies of the cores of the 2-handles. The surface $S_{\beta_{n,m},g}'$ has genus $(n-1)(m-1)+g$ and $[Z_{n,m}] = (n,m) \in H_2(S^2 \times S^2;\mathbb{Z})$. For the negative classes, we apply diffeomorphisms of $S^2 \times S^2$ that reverse orientation in the relevant factors of $S^2 \times S^2$ to the surface $Z_{n,m}$. Our convention shall be to drop the subscript $g$. 

\begin{rmk} \label{orbitsurface}
By construction, the surfaces representing negative classes in $H_2(S^2\times S^2;\mathbb{Z})$ are of the form $\varphi(Z_{n,m})$ for a diffeomorphism $\varphi : S^2 \times S^2 \to S^2 \times S^2$ and $n,m >0$. In particular, if $Z_{n,m}$ is (spin) flexible, so too is $\varphi(Z_{n,m})$. 
\end{rmk}

\textbf{Flexibility}. Recall that a link in $S^3$ is \textit{proper} if the linking number of any component  with the remaining components is even. In our examples, flexibility arises from the failure of the link in the boundary of the 0-handle to be proper (we refer the reader to Figure \ref{slide}). For a Seifert surface $S \subset S^3$, recall that $\theta_S : H_1(S;\mathbb{Z}/2\mathbb{Z}) \to \mathbb{Z}/2\mathbb{Z}$ denotes the mod two linking form on $S$.  

\begin{proposition} \label{geometricwitness}
Let $\Sigma \subset M^4$ be a closed surface in a 4-manifold $M$. Assume that there exists a 4-ball $B^4 \subset M$ for which $\Sigma \cap B^4 = \Sigma \cap \partial B^4 = S$, a Seifert surface for a non-proper link in $S^3$. Assume that $\Sigma \backslash \textrm{int}(S)$ is a collection of closed 2-disks. Then there is a component $\Delta$ of $\partial S$ such that isotoping a simple closed curve $\gamma \subset S$ across the disk that $\Delta$ bounds yields a simple closed curve $\widetilde{\gamma} \subset S$ for which \[\theta_{S}(\widetilde{\gamma}) = \theta_S(\gamma) +1 \mod 2. \]
\end{proposition}

\begin{proof}
By assumption, $\partial S$ is not proper. Let $\Delta$ denote a component of $\partial S$ whose linking number with the remaining components $\partial S -\Delta$ is odd. Then 
\[ 0 = \mathrm{lk}(\partial S,\Delta^+) = \mathrm{lk}(\partial S-\Delta,\Delta^+) + \mathrm{lk}(\Delta,\Delta^+) = \mathrm{lk}(\partial S-\Delta,\Delta) + \mathrm{lk}(\Delta,\Delta^+).\]
Therefore $\theta_S(\Delta) = \mathrm{lk}(\Delta,\Delta^+)$ is odd. By construction, the curves $\gamma \cup \widetilde{\gamma} \cup \Delta$ bound a pair of pants on $S$ and so 
\[ \theta_S(\gamma) + \theta_S(\widetilde{\gamma}) + \theta_S(\Delta) = 0 \mod 2.\]
Therefore $\theta_{S}(\widetilde{\gamma}) = \theta_S(\gamma) +1 \mod 2.$
\end{proof}

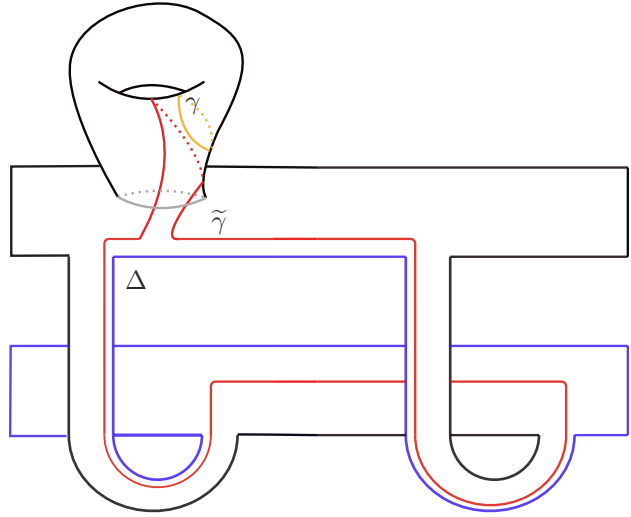
\begin{figure}
  \centering
  \def\svgwidth{0.55\columnwidth}
  %% Creator: Inkscape 1.4 (e7c3feb1, 2024-10-09), www.inkscape.org
%% Accompanies image file 'SlideOverDisk.pdf' (pdf, eps, ps)
\begingroup%
  \makeatletter%
  \providecommand\color[2][]{%
    \errmessage{(Inkscape) Color is used for the text in Inkscape, but the package 'color.sty' is not loaded}%
    \renewcommand\color[2][]{}%
  }%
  \providecommand\transparent[1]{%
    \errmessage{(Inkscape) Transparency is used (non-zero) for the text in Inkscape, but the package 'transparent.sty' is not loaded}%
    \renewcommand\transparent[1]{}%
  }%
  \providecommand\rotatebox[2]{#2}%
  \newcommand*\fsize{\dimexpr\f@size pt\relax}%
  \newcommand*\lineheight[1]{\fontsize{\fsize}{#1\fsize}\selectfont}%
  \ifx\svgwidth\undefined%
    \setlength{\unitlength}{197.89259783bp}%
    \ifx\svgscale\undefined%
      \relax%
    \else%
      \setlength{\unitlength}{\unitlength * \real{\svgscale}}%
    \fi%
  \else%
    \setlength{\unitlength}{\svgwidth}%
  \fi%
  \global\let\svgwidth\undefined%
  \global\let\svgscale\undefined%
  \makeatother%
  \begin{picture}(1,0.81902259)%
    \lineheight{1}%
    \setlength\tabcolsep{0pt}%
    \put(0,0){\includegraphics[width=\unitlength,page=1]{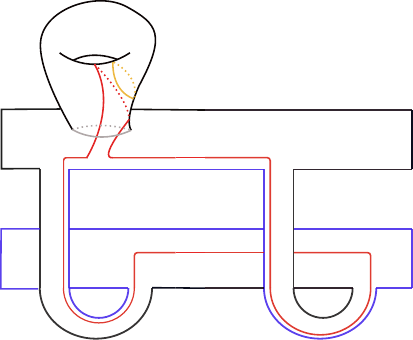}}%
    \put(0.28476511,0.65071971){\color[rgb]{0.01568627,0.01176471,0.00392157}\transparent{0.85097998}\makebox(0,0)[lt]{\lineheight{1.25}\smash{\begin{tabular}[t]{l}$\gamma$ \end{tabular}}}}%
    \put(0.32669918,0.45739489){\color[rgb]{0.01568627,0.01176471,0.00392157}\transparent{0.85097998}\makebox(0,0)[lt]{\lineheight{1.25}\smash{\begin{tabular}[t]{l}$\widetilde{\gamma}$\end{tabular}}}}%
    \put(0.1871581,0.36016168){\color[rgb]{0.01568627,0.01176471,0.00392157}\transparent{0.85097998}\makebox(0,0)[lt]{\lineheight{1.25}\smash{\begin{tabular}[t]{l}$\Delta$\end{tabular}}}}%
  \end{picture}%
\endgroup%

  \caption{\textnormal{Consider $S_{\beta_2,1}\subset S^3$, the connect sum of a Hopf band and unknotted torus, a subsurface of $Z_2 \subset \mathbb{C}P^2$. By construction, the boundary components of the Hopf band bound disks in the 2-handle. Sliding  $\gamma$ across one of these disks yields a curve $\widetilde{\gamma}$ on $S_{\beta_2,1}$, where $\mathrm{lk}(\gamma,\gamma^+)  = 0$ and $\mathrm{lk}(\widetilde{\gamma},\widetilde{\gamma}^+) = 1 \mod 2$ (in fact $\widetilde{\gamma}$ is a Hopf band curve).  Sliding a standard pair of dual curves on the torus across disks in the 2-handle yields a pair of dual Hopf band curves. Applying Proposition \ref{HiroseHopfBand} yields flexibility of $Z_2$}. }
  \label{slide}
\end{figure}

We are now in a position to prove

\setcounter{maintheorem}{0}
\begin{maintheorem} \label{thmA}
Let $M = \mathbb{C}P^2$ or $S^2 \times S^2$. Fix $x \in H_2(M;\mathbb{Z})$. If $x$ is characteristic, then $x$ admits a spin-flexible representative, otherwise $x$ admits a flexible representative. 
\end{maintheorem}

\begin{proof}
If $x$ is the trivial class, we may use a genus zero representative. If $M = S^2 \times S^2$, it suffices to work with positive classes (Remark \ref{orbitsurface}). Given $x \in H_2(M;\mathbb{Z})$, let $Z_x \subset M^4$ be a representative for $x$ constructed in Section \ref{ConstructingSurfaces} with $Z_x \cap B^4 = S_{\beta_x,g}'$ and $g \geq 5$. We start by assuming  that $x$ is characteristic. Then $\mathcal{E}(M^4,Z_x) \leq \Mod(Z_x)[q_{Z_x}]$ where $q_{Z_x} : H_1(Z_x;\mathbb{Z}/2\mathbb{Z}) \to \mathbb{Z}/2\mathbb{Z}$ is Rochlin's quadratic form. Recall that a simple closed curve $\gamma \subset Z_x$ is admissible if it is nonseparating and $q_{Z_x}(\gamma)  = 1$ (Remark \ref{admissibleunderform}). By Theorem \ref{admissiblespin} and Proposition \ref{QuadraticFormSpin}, \[\Mod(Z_x)[q_{Z_x}] = \langle T_\gamma \,|\, \gamma \subset Z_x \textrm{ admissible} \rangle.\] Fix an admissible curve $\gamma \subset Z_x$. By an isotopy of $\gamma$ if necessary, we may assume that $\gamma \subset S_{\beta_x,g}'$. Then \[ q_{S_{\beta_x,g}'}(\gamma) = q_{Z_x}(\gamma) = 1.\] By Lemma \ref{twistson1curves},  $T_{\gamma} \in \mathcal{E}_{\partial}(B^4,S_{\beta_x,g}')$. Therefore $T_{\gamma} \in \mathcal{E}(M,Z_x)$ and  $\mathcal{E}(M,Z_x) =  \Mod(Z_x)[q_{Z_x}]$.

Assume now that $x$ is not a characteristic element. To prove that $\mathcal{E}(M,Z_x) = \Mod(Z_x)$, it suffices, in light of Lemma \ref{twistson1curves}, to prove that if $\gamma \subset S_{\beta_x,g}'$ is a nonseparating simple closed curve with $q_{S_{\beta_x,g}'}(\gamma) = 0$, then $T_{\gamma} \in \mathcal{E}(M,Z_x)$. By inspection, the link $\partial S_{\beta_x,g}'$ is not proper. By Proposition \ref{geometricwitness}, there exists a boundary component $\Delta$ of $S_{\beta_x,g}'$ for which isotoping $\gamma$ across the disk that $\Delta$ bounds yields a curve $\widetilde{\gamma} \subset S_{\beta_x,g}'$ with \[q_{S_{\beta_x,g}'}(\widetilde{\gamma}) = 1.\] By Lemma \ref{twistson1curves}, $T_{\widetilde{\gamma}} \in \mathcal{E}_{\partial}(B^4,S_{\beta_x,g}')$. As $\gamma$ and $\widetilde{\gamma}$ are isotopic on $Z_x$, $T_{\gamma} = T_{\widetilde{\gamma}} \in \mathcal{E}(M,Z_x)$. As $\Mod(Z_x)$ is generated by Dehn twists on nonseparating simple closed curves, $Z_x$ is a flexible surface. 
\end{proof}

\section{A Single stabilization for smooth plane curves}

Let $X_d \subset \mathbb{C}P^2$ be a smooth plane curve of degree $d$. Recall that any two such curves are smoothly isotopic in $\mathbb{C}P^2$. Again, we consider the standard handle decomposition of $\mathbb{C}P^2$ with 0-handle $B^4 \subset \mathbb{C}P^2$. Following \cite{AK} and \cite{GS}, up to isotopy, we may assume $X_d \cap B^4 = X_d \cap \partial B^4 = F_{\beta_d}$ where $\beta_d = (\sigma_1\cdots \sigma_{d-1})^d$, and $F_{\beta_d}$ is the Seifert surface for the closure of $\beta_d$ constructed in Section 4. Furthermore, $X_d = F_{\beta_d} \cup D$ where $D$ is a disjoint collection of 2-disks (parallel copies of the core of the 2-handle).

Given this description of smooth plane curves, the methods used to establish Theorem \ref{thmA} prove that stabilizing $X_d$ with an unknot in $S^4$ of genus at least 5 yields a (spin) flexible surface. We now show that a single stabilization suffices, starting with a general argument for curves of degree $d \geq 5$ and a sequence of ad hoc arguments dealing with $d \leq 4$. We continue to use notation and conventions established in Subsection \ref{integralliftsection}. 

Let $Z_d = X_d \# T \subset \mathbb{C}P^2 \# S^4$ where $T \subset S^4$ is an unknotted torus. Up to isotopy, we may assume $Z_d \cap \partial B^4 = S_{\beta_d,1}$ where $S_{\beta_d,1} = F_{\beta_d} \# T$. Recall that we have equipped $F_{\beta_d}$ with a framing $\phi_{\beta_d} : S(F_{\beta_d}) \to \mathbb{Z}$, which induces a framing on $\widehat{F}_{\beta_d} = F_{\beta_d} - \textrm{int}(D_0)$, where $D_0$ is the surgery disk for the connect sum $F_{\beta_d} \# T$.

\begin{rmk}
In clarifying remarks made from the introduction, we note that there is a winding number function $\phi_d : \mathcal{S}(X_d) \to \mathbb{Z}/(d-3)$ arising from a $(d-3)^{rd}$ root of the adjoint line bundle of $X$ in $\mathbb{C}P^2$ (we refer the reader to \cite{SalterToric} for more details). The framing $\phi_{\beta_d}$ induces a $(d-3)$-spin structure on $X_d$ which \textit{coincides} with $\phi_d$. In this case, $\mathcal{C}_{F_{\beta_d}}$ consists of \textit{vanishing cycles} for nodal degenerations of the plane curve $X_d$. We do not need these facts and so we omit details. 
\end{rmk}

Recall that $P_{\beta_d,1} = S_{\beta_d,1} - \textrm{int}(D_1)$ where $D_1$ is a closed 2-disk supported in the torus $T$. We extend the framing $\phi_{\beta_d}$ to a framing $\phi^+_{\beta_d}$ of $P_{\beta_d,1}$ by declaring the curves $\delta_1,\delta_2 \subset P$ depicted in Figure \ref{DeltaCurvesThmB} admissible (Lemma \ref{extendframing}).  By construction, this extension is an integral lift of the linking form on the Seifert surface $S_{\beta_d,1}$.

\begin{figure}
  \centering
  \def\svgwidth{0.45\columnwidth}
  %% Creator: Inkscape 1.4 (e7c3feb1, 2024-10-09), www.inkscape.org
%% Accompanies image file 'PairofDeltaCurvesThmB.pdf.pdf' (pdf, eps, ps)
\begingroup%
  \makeatletter%
  \providecommand\color[2][]{%
    \errmessage{(Inkscape) Color is used for the text in Inkscape, but the package 'color.sty' is not loaded}%
    \renewcommand\color[2][]{}%
  }%
  \providecommand\transparent[1]{%
    \errmessage{(Inkscape) Transparency is used (non-zero) for the text in Inkscape, but the package 'transparent.sty' is not loaded}%
    \renewcommand\transparent[1]{}%
  }%
  \providecommand\rotatebox[2]{#2}%
  \newcommand*\fsize{\dimexpr\f@size pt\relax}%
  \newcommand*\lineheight[1]{\fontsize{\fsize}{#1\fsize}\selectfont}%
  \ifx\svgwidth\undefined%
    \setlength{\unitlength}{179.51387661bp}%
    \ifx\svgscale\undefined%
      \relax%
    \else%
      \setlength{\unitlength}{\unitlength * \real{\svgscale}}%
    \fi%
  \else%
    \setlength{\unitlength}{\svgwidth}%
  \fi%
  \global\let\svgwidth\undefined%
  \global\let\svgscale\undefined%
  \makeatother%
  \begin{picture}(1,0.90288629)%
    \lineheight{1}%
    \setlength\tabcolsep{0pt}%
    \put(0,0){\includegraphics[width=\unitlength,page=1]{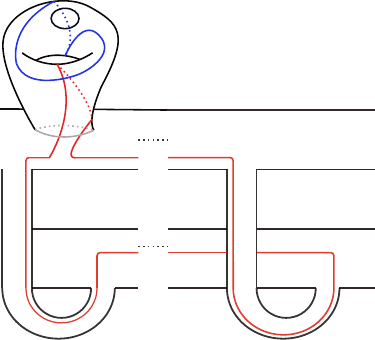}}%
    \put(0.04174416,0.65780825){\color[rgb]{0.03921569,0.04705882,0.10980392}\makebox(0,0)[lt]{\lineheight{1.25}\smash{\begin{tabular}[t]{l}$\delta_2$\end{tabular}}}}%
    \put(0.2323051,0.50399432){\color[rgb]{0.03921569,0.04705882,0.10980392}\makebox(0,0)[lt]{\lineheight{1.25}\smash{\begin{tabular}[t]{l}$\delta_1$\end{tabular}}}}%
  \end{picture}%
\endgroup%

  \caption{\textnormal{The framing $\phi_{\beta_d}$ on $\widehat{F}_{\beta_d}$ is extended over the surface $P_{\beta_d,1} = S_{\beta_d,1}-\textrm{int}(D_1)$ by declaring the curves $\delta_1,\delta_2$ admissible (winding number zero). Both $\delta_1,\delta_2$ are Hopf band curves.  }}
  \label{DeltaCurvesThmB}
\end{figure}

In Section 4, we relied on a sufficiently large unknotted piece of the surface to initiate an inductive framework (Lemma \ref{stabilization}). This propagated into a lower bound on the extendable subgroup. We now rely on the \textit{geometric monodromy group} of the braid $\beta_d$, using work of Ferretti in \cite{Ferretti}. Recall, from Section 3, that on the Seifert surface $F_{\beta_d} \subset S^3$ there is the \textit{braid word assemblage} $\mathcal{C}_{F_{\beta_d}}$ (a filling configuration of Hopf band curves). The geometric monodromy group of the braid $\beta_d$ is \[\mathcal{T}_{\mathcal{C}_{F_{\beta_d}}} = \langle T_c \,:\, c\in \mathcal{C}_{F_\beta} \rangle.\] 

\begin{lemma} \label{genframedmcg}
For $d \geq 5$, $\Mod(F_{\beta_d})[\phi_{\beta_d}] = \langle T_c \,:\, c\in \mathcal{C}_{F_\beta} \rangle$. 
\end{lemma}

\begin{proof}
This is due to Ferretti following from \cite[Theorem 2]{Ferretti} supplemented with \cite[Remark 5.1]{Ferretti}. It also follows directly from \cite[Theorem 2]{Ferretti} as follows. Let $\nu = (\sigma_1\cdots\sigma_{d-1})^{d-1}$. Then $(F_\nu,\phi_\nu) \subset (F_{\beta_d},\phi_{\beta_d})$ (an inclusion of framed surfaces) and $F_\nu$ is a Seifert surface for the $(d,d-1)$ torus knot. By \cite[Theorem 2]{Ferretti},  $\Mod(F_{\nu})[\phi_\nu] = \langle T_c \,:\, c \in \mathcal{C}_{F_\nu} \rangle$. The surface $F_{\beta_d}$ is obtained by stabilizing $F_\nu$ along the collection of admissible curves in $\mathcal{C}_{F_\beta} \backslash C_{F_\nu}$. As the genus of $F_\nu$ is at least 5, Lemma \ref{stabilization} yields $\Mod(F_\beta)[\phi_\beta] = \langle T_c \,:\, c\in \mathcal{C}_{F_\beta} \rangle$. 
\end{proof}

Henceforth suppress all subscripts, writing $F = F_{\beta_d}$ , $S = F_{\beta_d}\# T$, $P = P_{\beta_d,1}$ and $\phi^+ = \phi^+_{\beta_d}$. Stabilization (Lemma \ref{stabilization}) gives 

\begin{proposition} \label{stabilizealonggenus1}
For $d \geq 5$ we have $\Mod(P)[\phi^+] = \langle \Mod(\widehat{F})[\phi], T_{\delta_1},T_{\delta_2}\rangle$.
\end{proposition}

Let $p : \Mod(P) \to \Mod(Z_d)$ be the homomorphism induced by extension as the identity. Recall that $Z_d\backslash P$ is a collection of disks.

\begin{proposition} \label{framedchange}
For all $d \geq 5$, $p(\Mod(P)[\phi^+]) \leq \mathcal{E}(\mathbb{C}P^2,Z_d)$. 
\end{proposition}

\begin{proof}
By Proposition \ref{stabilizealonggenus1}, $\Mod(P)[\phi^+] = \langle \Mod(\widehat{F})[\phi], T_{\delta_1},T_{\delta_2} \rangle$. As $\delta_1$ and $\delta_2$ are Hopf band curves, it follows from Proposition \ref{HiroseHopfBand} that their image under $p$ is in $\mathcal{E}(\mathbb{C}P^2,Z_d)$. It remains to prove that \[p(\Mod(\widehat{F})[\phi]) \leq \mathcal{E}(\mathbb{C}P^2,Z_d).\]  Capping $\widehat{F}$ with $D_0$ and extending mapping classes by the identity over $F$ induces a short exact sequence 
$$ 1 \to \pi_1(UTF) \to \Mod(\widehat{F}) \to \Mod(F) \to 1,$$ 
where $UTF$ denotes the unit tangent bundle of $F$.  If  $\widetilde{\gamma} \in \pi_1(UTF)$ has image $\gamma \in \pi_1(F)$ (based at point in $D_0$) under projection $UTF\to F$, the image of $\widetilde{\gamma}$ in $\Mod(\widehat{F})$ is the multitwist $T_{\gamma_L}T_{\gamma_R}^{-1}T_{\partial D_0}^k$ for some $k \in \mathbb{Z}$, where $\gamma_L$ (resp. $\gamma_R$) is obtained by pushing $\gamma$ to the left (resp. right) over $D_0$. 

Using the generating set established in Lemma \ref{genframedmcg}, we see that capping is surjective on the level of framed mapping class groups, and we have a short exact sequence $$1 \to \pi_1(UTF) \cap \Mod(\widehat{F})[\phi] \to \Mod(\widehat{F})[\phi] \to \Mod(F)[\phi]\to 1.$$  The group $\Mod(\widehat{F})[\phi]$ is generated by lifting a generating set for $\Mod(F)[\phi]$ together with $\pi_1(UTF) \cap \Mod(\widehat{F})[\phi]$. The generating set for $\Mod(F)[\phi]$ supplied by Lemma \ref{genframedmcg} consists of Dehn twists about Hopf band curves supported in $\widehat{F}$. It remains to now show \[p(\pi_1(UTF) \cap \Mod(\widehat{F})[\phi]) \leq \mathcal{E}(\mathbb{C}P^2,Z_d).\] In fact we prove \[p(\pi_1(UTF))\leq \mathcal{E}(\mathbb{C}P^2,Z_d).\] Choose a generating set $\{\gamma_k\}_{k=1}^{n_d}$ for $\pi_1(F)$ consisting of simple closed curves based at a point on $ \partial D_0$, each of which is a Hopf band curve on $F$. We recall that the image of $\pi_1(UTF)$ in $\Mod(\widehat{F})$ is generated by $\textrm{Push}(\gamma_k) = T_{\gamma_k,L}T_{\gamma_k,R}^{-1}$ for $k = 1,\ldots, n_d$ together with $T_{\partial D_0}$. Both $\gamma_{k,L}$ and $\gamma_{k,R}$ are Hopf band curves so $p_Z(\textrm{Push}(\gamma_k)) \in \mathcal{E}(\mathbb{C}P^2,Z_d)$, and by Proposition \ref{separatingunlink}, $p_Z(T_{\Delta}) \in \mathcal{E}(\mathbb{C}P^2,Z_d)$.
\end{proof}

We now exploit the existence of a compression curve on the surface $S$ to manufacture more extendable mapping classes in $\mathcal{E}(\mathbb{C}P^2,Z_d)$. Such a curve does not exist on the surface $X_d$,  as it minimizes genus in its homology class (Thom Conjecture, \cite{KronMrowka}).

\begin{proposition} \label{existcurves}
There exists a nonseparating simple closed curve $\alpha \subset \widehat{F} \subset S$ for which $\phi(\alpha) = \pm2$ and $p(T_{\alpha})\in\mathcal{E}(\mathbb{C}P^2,Z_d)$. 
\end{proposition}

\begin{proof}
Consider the oriented curves $\alpha \subset \widehat{F}$ and $\delta \subset S$ depicted in Figure \ref{existenceofcurves} with $\phi(\alpha) = 2$ and $\mathrm{lk}(\alpha,\alpha^+) = -3$. The curve $\delta$ is an unknot with $\textrm{lk}(\delta,\delta^+) = +1$ i.e. a Hopf band curve (see Definition \ref{Hopfbandcurve}). By Proposition \ref{HiroseHopfBand}, $T_{\delta} \in \mathcal{E}(\mathbb{C}P^2,Z_d)$. %We compute $\mathrm{lk}(\alpha,\delta^+) = +1$ and $\mathrm{lk}(\delta,\alpha^+) = 0$.
As $[T_{\delta}^{-1}(\alpha)] = [\alpha] + [\delta] \in H_1(F;\mathbb{Z})$ it follows that \[\textrm{lk}(T_{\delta}^{-1}(\alpha),T_{\delta}^{-1}(\alpha)^+) = \mathrm{lk}(\alpha,\alpha^+) + \mathrm{lk}(\alpha,\delta^+) + \mathrm{lk}(\delta,\alpha^+) + \mathrm{lk}(\delta,\delta^+) = -3 + 1 + 0 + 1 =  -1.\] By inspection, $T_{\delta}^{-1}(\alpha)$ is an unknot and therefore a Hopf band curve. By Proposition \ref{HiroseHopfBand}, we have $T_{T_{\delta}^{-1}(\alpha)} \in \mathcal{E}(\mathbb{C}P^2,Z_d)$. In $\Mod(Z_d)$, we have \[T_{T_{\delta}^{-1}(\alpha)} = T_{\delta}^{-1}T_{\alpha}T_{\delta}, \] and so $p(T_{\alpha}) \in \mathcal{E}(\mathbb{C}P^2,Z_d)$. 
\end{proof}

\begin{figure}
  \centering
  \def\svgwidth{0.9\columnwidth}
  %% Creator: Inkscape 1.4 (e7c3feb1, 2024-10-09), www.inkscape.org
%% Accompanies image file 'UntwistingCurves.pdf' (pdf, eps, ps)
\begingroup%
  \makeatletter%
  \providecommand\color[2][]{%
    \errmessage{(Inkscape) Color is used for the text in Inkscape, but the package 'color.sty' is not loaded}%
    \renewcommand\color[2][]{}%
  }%
  \providecommand\transparent[1]{%
    \errmessage{(Inkscape) Transparency is used (non-zero) for the text in Inkscape, but the package 'transparent.sty' is not loaded}%
    \renewcommand\transparent[1]{}%
  }%
  \providecommand\rotatebox[2]{#2}%
  \newcommand*\fsize{\dimexpr\f@size pt\relax}%
  \newcommand*\lineheight[1]{\fontsize{\fsize}{#1\fsize}\selectfont}%
  \ifx\svgwidth\undefined%
    \setlength{\unitlength}{539.36837264bp}%
    \ifx\svgscale\undefined%
      \relax%
    \else%
      \setlength{\unitlength}{\unitlength * \real{\svgscale}}%
    \fi%
  \else%
    \setlength{\unitlength}{\svgwidth}%
  \fi%
  \global\let\svgwidth\undefined%
  \global\let\svgscale\undefined%
  \makeatother%
  \begin{picture}(1,0.49939142)%
    \lineheight{1}%
    \setlength\tabcolsep{0pt}%
    \put(0,0){\includegraphics[width=\unitlength,page=1]{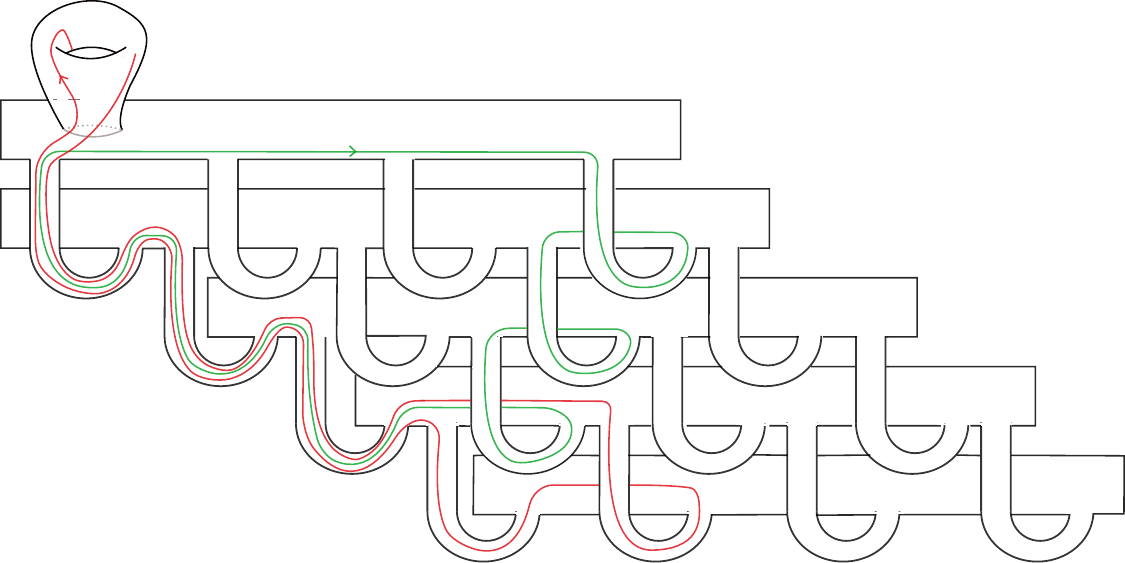}}%
    \put(0.3671256,0.37419762){\color[rgb]{0.07058824,0.02745098,0.02745098}\transparent{0.85097998}\makebox(0,0)[lt]{\lineheight{1.25}\smash{\begin{tabular}[t]{l}$\alpha$\end{tabular}}}}%
    \put(0.00229204,0.38842256){\color[rgb]{0.07058824,0.02745098,0.02745098}\transparent{0.85097998}\makebox(0,0)[lt]{\lineheight{1.25}\smash{\begin{tabular}[t]{l}$\delta$\end{tabular}}}}%
    \put(0,0){\includegraphics[width=\unitlength,page=2]{UntwistingCurves.pdf}}%
  \end{picture}%
\endgroup%

  \caption{\textnormal{As depicted the calculation takes place on $S_{\nu,1}$ for $\nu= (\sigma_1\sigma_2\sigma_3\sigma_4)^4$, which is naturally a subsurface of $S_{\beta_d,1}$ for all $d \geq 5$.}}
  \label{existenceofcurves}
\end{figure}

We now use the  framed change-of-coordinates principle (Proposition \ref{framedchange}) to deal with arbitrary nonseparating simple closed curves.

\begin{lemma}\label{untwist}
Let $\gamma \subset P \subset Z_d = X_d \# T$ be a nonseparating simple closed curve with $\phi^+(\gamma) = 0 \mod 2$.  Then $p(T_{\gamma})\in \mathcal{E}(\mathbb{C}P^2,Z_d)$. If $d$ is even, then for any nonseparating simple closed curve $\delta \subset P$, we have $p(T_\delta) \in \mathcal{E}(\mathbb{C}P^2,Z_d)$.  
\end{lemma}

\begin{proof}
Orient $\gamma$ and write $\phi^+(\gamma) = 2k$. By Proposition \ref{existcurves},  there exists a nonseparating simple closed curve $\alpha \subset \widehat{F} \subset P$ with $\phi^+(\alpha) = 2$ and $p(T_{\alpha}) \in \mathcal{E}(\mathbb{C}P^2,Z_d)$. Let $\beta$ be a geometrically dual admissible curve to $\alpha$ (by Lemma \ref{dualcurve}). Let $\mu = T^k_\alpha(\beta)$. By twist-linearity, $\phi^+(\mu) = 2k$, and \[T_{\mu} = T_{\alpha}^k T_\beta T_{\alpha}^{-k}.\]By Proposition \ref{framedchange}, $p(T_\beta) \in \mathcal{E}(\mathbb{C}P^2,Z_d)$ and so $T_\mu \in \mathcal{E}(\mathbb{C}P^2,Z_d)$. By the framed change-of-coordinates principle on $P$ (Lemma \ref{framedchange}), there exists $f \in \Mod(P)[\phi^+]$ with $f(\mu) = \gamma$. Now $T_{\gamma} = fT_{\mu}f^{-1}$, and by Corollary \ref{framedchange},  $p(f) \in \mathcal{E}(\mathbb{C}P^2,Z_d)$. It follows that $p(T_{\gamma}) \in \mathcal{E}(\mathbb{C}P^2,Z_d)$.

Assume now that $d$ is even. By the above, we may assume that $\phi^+(\delta) = 1 \mod 2$. By Proposition \ref{geometricwitness}, we may isotope $\delta$ on $Z_d$ across a disk in $D$ (a parallel copy of the core of the 2-handle of $\mathbb{C}P^2$) obtaining  $\widetilde{\delta}$ with $\phi^+(\widetilde{\delta}) = 0 \mod 2$. As $\delta$ and $\widetilde{\delta}$ are isotopic curves on $Z_d$, we conclude $p_{Z_d}(T_\delta) = p_{Z_d}(T_{\widetilde{\delta}}) \in \mathcal{E}(\mathbb{C}P^2,Z_d)$. 
\end{proof}

\begin{figure}
  \centering
  \def\svgwidth{0.9\columnwidth}
  %% Creator: Inkscape 1.4 (e7c3feb1, 2024-10-09), www.inkscape.org
%% Accompanies image file 'Degree4Curves.pdf' (pdf, eps, ps)
\begingroup%
  \makeatletter%
  \providecommand\color[2][]{%
    \errmessage{(Inkscape) Color is used for the text in Inkscape, but the package 'color.sty' is not loaded}%
    \renewcommand\color[2][]{}%
  }%
  \providecommand\transparent[1]{%
    \errmessage{(Inkscape) Transparency is used (non-zero) for the text in Inkscape, but the package 'transparent.sty' is not loaded}%
    \renewcommand\transparent[1]{}%
  }%
  \providecommand\rotatebox[2]{#2}%
  \newcommand*\fsize{\dimexpr\f@size pt\relax}%
  \newcommand*\lineheight[1]{\fontsize{\fsize}{#1\fsize}\selectfont}%
  \ifx\svgwidth\undefined%
    \setlength{\unitlength}{496.86521035bp}%
    \ifx\svgscale\undefined%
      \relax%
    \else%
      \setlength{\unitlength}{\unitlength * \real{\svgscale}}%
    \fi%
  \else%
    \setlength{\unitlength}{\svgwidth}%
  \fi%
  \global\let\svgwidth\undefined%
  \global\let\svgscale\undefined%
  \makeatother%
  \begin{picture}(1,0.48014189)%
    \lineheight{1}%
    \setlength\tabcolsep{0pt}%
    \put(0,0){\includegraphics[width=\unitlength,page=1]{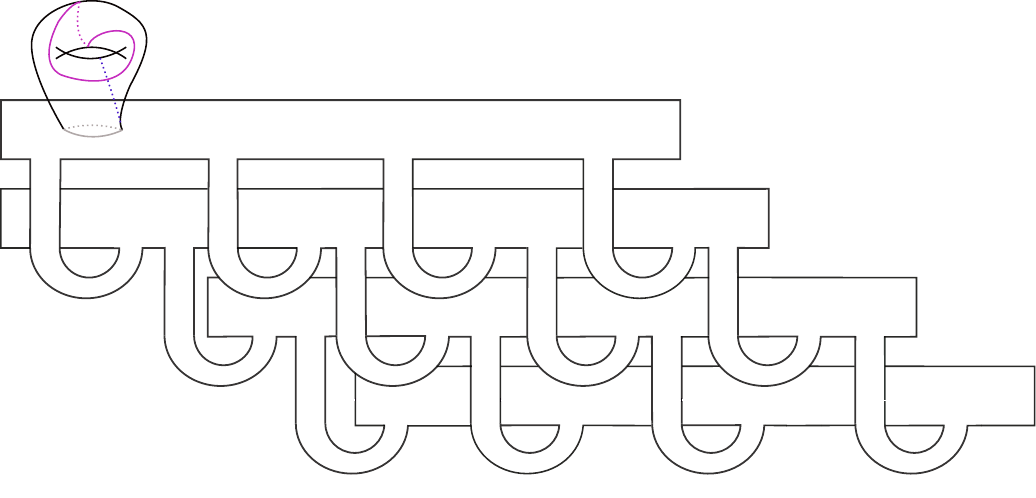}}%
    \put(0.43863311,0.36157189){\color[rgb]{0.05098039,0.02352941,0.03921569}\makebox(0,0)[lt]{\lineheight{1.25}\smash{\begin{tabular}[t]{l}$\alpha$ \end{tabular}}}}%
    \put(0,0){\includegraphics[width=\unitlength,page=2]{Degree4Curves.pdf}}%
  \end{picture}%
\endgroup%

  \caption{\textnormal{Dehn twists on the depicted configuration of curves yields a generating set for $\Mod(Z_4)$. The configuration is equivalent to the standard Humphries generating set for the mapping class group of a closed surface of genus four \cite{Primer}. }}
  \label{Degree4}
\end{figure}

Having constructed an ample supply of Dehn twists in $\mathcal{E}(\mathbb{C}P^2,Z_d)$ we prove Theorem \ref{ThmB}. Recall that if $d$ is odd $Z_d = X_d \# T \subset \mathbb{C}P^2$ is characteristic, $q_{Z_d} : H_1(Z_d;\mathbb{Z}/2\mathbb{Z}) \to \mathbb{Z}/2\mathbb{Z}$ denotes Rochlin's quadratic form, and $Z_d\backslash P$ is a collection of 2-disks. 

\begin{maintheorem} \label{ThmB}
Let $X_d \subset \mathbb{C}P^2$ be a smooth plane curve of degree $d$.  Let $T \subset S^4$ be the boundary of a solid torus in $S^4$. Let $Z_d$ be the stabilization $X_d \# T \subset \mathbb{C}P^2 \# S^4 \cong \mathbb{C}P^2$. If $d$ is even, $Z_d$ is flexible. If $d$ is odd,  $Z_d$ is spin-flexible. 
\end{maintheorem}

\begin{proof}
We start with the general case $d \geq 5$. Assume that $d$ is even. Let $\gamma \subset Z_d$ be a nonseparating simple closed curve. By isotopy if necessary, we may assume that $\gamma \subset P$. Then \[\phi^+(\gamma) = 0 \textrm{ or } 1 \mod 2.\] In either case, by Lemma \ref{untwist}, it follows that $T_{\gamma} \in \mathcal{E}(\mathbb{C}P^2,Z_d)$. As $\Mod(Z_d)$ is generated by Dehn twists on nonseparating simple closed curves,  $Z_d$ is a flexible surface.

If $d$ is odd, there is the upper bound \[\mathcal{E}(\mathbb{C}P^2,Z_d) \leq \Mod(Z_d)[q_{Z_d}].\] By Theorem \ref{admissiblespin}, $\Mod(Z_d)[q_{Z_d}]$ is generated by Dehn twists on nonseparating simple closed curves $\gamma \subset Z_d$ with $q_{Z_d}(\gamma) = 1$. Fix such a curve $\gamma$. We may assume $\gamma \subset P$. If $\theta_{S}$ denotes the linking form on $H_1(S;\mathbb{Z}/2\mathbb{Z})$ as a Seifert surface in $S^3$, then $i^*q_{Z} = \theta_S$ where $i : S \to Z_d$ denotes inclusion (Proposition \ref{Rochlinislink}).  It follows that $\mathrm{lk}(\gamma,\gamma^+) = 1 \mod 2$ and so $\phi^+(\gamma) = 0 \mod 2$. By Lemma \ref{untwist}, $T_{\gamma}\in \mathcal{E}(\mathbb{C}P^2,Z_d)$, and $Z_d$ is spin-flexible if $d$ is odd.

We now deal with the remaining low degree cases. If $d = 1$, spin flexibility follows from Proposition \ref{relboundaryhirose}. If $d = 2$, flexibility of $Z_2= X_2 \# \Sigma_1$ follows from Figure \ref{slide}. If $d = 3$, the surface $Z_3 = X_3 \# \Sigma_1$ is of genus 2 and $q_{Z_3}$ is an \textit{odd} (Arf invariant 1) quadratic form (Theorem \ref{ArfRochlin}). A generating set for the stabilizer $\Mod(Z_3)[q_{Z_3}]$ is given by a chain of four nonseparating admissible simple closed curves (see \cite[Lemma 4.4]{Hamenstadt}), and such a configuration of Hopf band curves on $S_{\beta_3,1} \subset S^3$ is easily constructed. Finally, consider $Z_4 = X_4 \# \Sigma_1$ of genus four. In Figure \ref{Degree4}, a Humphries generating set for $\Mod(Z_4)$ is depicted. All but one curve is a Hopf band curve. The curve $\alpha$ has $\mathrm{lk}(\alpha,\alpha^+) = -3$ and arguing as in Proposition \ref{existcurves}, using a suitably dual positive Hopf band curve, we see that $T_{\delta} \in \mathcal{E}(\mathbb{C}P^2,Z_4)$.  
\end{proof}

\begin{rmk} \label{d=3difference}
Hirose proves \cite[Theorem 5.1]{HiroseCP^2} that, for an unknot $\Sigma_g \subset S^4$ of genus $g$, the surface $X_3 \# \Sigma_g$ is spin flexible for any $g \geq 1$. Whilst Theorem \ref{ThmB} generalizes this to any $d$ for $g = 1$, a straightforward generalization of the methods shows that the surface $X_d \# \Sigma_g$ is (spin) flexible for any $g \geq 1$.
\end{rmk}

\bibliographystyle{alpha} 
\bibliography{refs}

\end{document}